\documentclass{amsart}
\usepackage{amsmath,amssymb,amscd}
\usepackage[enableskew]{youngtab}
\usepackage{graphicx}
\usepackage[usenames]{color}

\newcommand{\bt}{\tilde{b}}
\newcommand{\geh}{\mathfrak{g}}
\newcommand{\is}{\mathrm{inner}}
\newcommand{\la}{\lambda}
\newcommand{\La}{\Lambda}
\newcommand{\ol}{\overline}
\newcommand{\os}{\mathrm{outer}}
\newcommand{\ot}{\otimes}
\newcommand{\ve}{\varepsilon}
\newcommand{\vp}{\varphi}
\newcommand{\wt}{\mathrm{wt}\,}
\newcommand{\xb}{\ol{x}}
\newcommand{\Z}{\mathbb{Z}}

\numberwithin{equation}{section}

\newtheorem{theorem}{Theorem}
\newtheorem{prop}[theorem]{Proposition}
\newtheorem{lemma}[theorem]{Lemma}
\newtheorem{cor}[theorem]{Corollary}

\theoremstyle{definition}

\newtheorem{remark}{Remark}
\newtheorem{example}[remark]{Example}

\numberwithin{theorem}{section}
\numberwithin{definition}{section}
\numberwithin{remark}{section}

\begin{document}

\title[Combinatorial $R$-matrices]
{Combinatorial $R$-matrices for Kirillov--Reshetikhin crystals of type $D^{(1)}_n,B^{(1)}_n,A^{(2)}_{2n-1}$}

\author[M.~Okado]{Masato Okado}
\address{Department of Mathematical Science,
Graduate School of Engineering Science, Osaka University,
Toyonaka, Osaka 560-8531, Japan}
\email{okado@sigmath.es.osaka-u.ac.jp}

\author[R.~Sakamoto]{Reiho Sakamoto}
\address{Department of Physics, Graduate School of Science, The University of Tokyo,
Hongo, Bunkyo-ku, Tokyo, 113-0033, Japan}
\email{reiho@spin.phys.s.u-tokyo.ac.jp}

%\thanks{\textit{Date:} February 27, 2009}
%\date{\today}

\begin{abstract}
We calculate the image of the combinatorial $R$-matrix for any classical highest weight element
in the tensor product of Kirillov--Reshetikhin crystals $B^{r,k}\ot B^{1,l}$ of type 
$D^{(1)}_n,B^{(1)}_n,A^{(2)}_{2n-1}$. The notion of $\pm$-diagrams is effectively used for the
identification of classical highest weight elements in $B^{1,l}\ot B^{r,k}$.
\end{abstract}

\maketitle

\section{Introduction}
Let $U'_q(\geh)$ be the quantum enveloping algebra associated to an affine algebra $\geh$ without 
derivation. Let $V,V'$ be finite-dimensional $U'_q(\geh)$-modules. Suppose $V\ot V'$ is irreducible
and $V,V'$ have crystal bases $B,B'$. Then it is known \cite{Ka3,O} that there exists a unique map
$R$ from $B\ot B'$ to $B'\ot B$ commuting with any crystal operators $e_i$ and $f_i$. There also exists
an integer-valued function $H$ on $B\ot B'$, called energy function, satisfying a certain recurrence
relation under the action of $e_i$ (see \eqref{eq:e-func}).

Combinatorial $R$-matrices or energy functions play an important role in the affine crystal theory.
In the Kyoto path model \cite{KMN1:1992}, that realizes the affine highest weight crystal in terms of a 
semi-infinite tensor product of perfect crystals, the energy function is an essential ingredient for 
the computation of the affine weight. In the box-ball system \cite{FOY,HHIKTT} or its generalizations 
\cite{HKOTY2} in the formulation of crystal bases, the time evolution of the system is defined by using 
the combinatorial $R$-matrix. Energy functions are also crucial in the calculation of conserved quantities.
In \cite{Sa} a new connection was revealed between the energy function and the KKR or KSS bijection 
\cite{KKR,KR1,KSS} that gives a one-to-one correspondence between highest weight paths and rigged 
configurations.

Recently, for all nonexceptional affine types, all KR crystals, crystal bases of Kirillov--Reshetikhin (KR)
modules (if they exist), were shown to exist and their combinatorial structures were clarified 
\cite{O:2007,S:2008,OS:2008,FOS}. Hence, it is natural to consider the problem of obtaining a rule to 
calculate the combinatorial $R$-matrix and energy function.

In this paper, for type $D^{(1)}_n,B^{(1)}_n,A^{(2)}_{2n-1}$ 
we calculate the image of the combinatorial $R$-matrix for any classical highest weight 
element in the tensor product of KR crystals $B^{r,k}\ot B^{1,l}$ (Theorem \ref{th:main}). (Note that the 
first upper index of the second component is 1.) We also obtain the value of the energy function for such 
elements. Although we get the rule only for highest weight elements, there is an advantage from the 
computational point of view, since it is always easy to calculate the action of crystal operators $e_i,f_i$ 
for $i\ne0$ not only by hand but also by computer. To identify highest weight elements in the image 
$B^{1,l}\ot B^{r,k}$ the notion of $\pm$-diagrams, introduced in \cite{S:2008}, is used effectively.

The paper is organized as follows. In Section 2 we briefly review crystals and $\pm$-diagrams. In Section 3
we recall the KR crystal $B^{r,k}$ for type $D^{(1)}_n,B^{(1)}_n$ and $A^{(2)}_{2n-1}$, and the notion of
combinatorial $R$-matrix and energy function. The condition for an element of $B^{r,k}\ot B^{1,l}$ or 
$B^{1,l}\ot B^{r,k}$ to be classically highest is also presented. The main theorem is given in Section 4.
In Section 5 we prove a special case of the theorem, and reduction to this case is discussed in Section 6
according to whether $r$ is odd or even.

\subsection*{Acknowledgements.}
MO was supported by grant JSPS 20540016. The work of RS is supported by
the Core Research for Evolutional Science and Technology
of Japan Science and Technology Agency.

\section{Reviews on crystals and $\pm$-diagrams}

\subsection{Crystals} \label{subsec:crystals}

Let $\geh$ stand for a simple Lie algebra or affine Kac--Moody Lie algebra with index set $I$ and 
$U_q(\geh)$ the corresponding quantized enveloping algebra. Axiomatically, a $\geh$-crystal is a 
nonempty set $B$ together with maps
\begin{equation*}
\begin{split}
	e_i, f_i &: B \to B \cup \{0\} \qquad \text{for $i\in I$,}\\
	\wt &: B \to P,
\end{split}
\end{equation*}
where $P$ is the weight lattice associated to $\geh$. The maps $e_i$ and $f_i$ are called Kashiwara
operators and $\wt$ is the weight function.
To each crystal one can associate a crystal graph with vertices in $B$ and an arrow colored $i\in I$
from $b$ to $b'$ if $f_ib=b'$ or equivalently $e_ib'=b$. For $b\in B$ and $i\in I$, let
\begin{equation*}
\begin{split}
	\ve_i(b) &= \max\{k \in \Z_{\ge 0} \mid e_i^kb \neq 0 \},\\
	\vp_i(b) &= \max\{k \in \Z_{\ge 0} \mid f_i^kb \neq 0 \}.
\end{split}
\end{equation*}
In this paper we only consider crystal bases coming from $U_q(\geh)$-modules. For a complete definition 
of crystal bases see for example~\cite{Ka:1991,HK:2002}.

Let $B_1,B_2$ be crystals. Then $B_1\ot B_2=\{b_1\ot b_2\mid b_1\in B_1,b_2\in B_2\}$ can be endowed with
the structure of crystal. The actions of Kashiwara operators and the value of the weight function are 
given by
\begin{align*}
e_i(b_1\ot b_2)&=\left\{
\begin{array}{ll}
e_ib_1\ot b_2\quad&\text{if }\vp_i(b_1)\ge\ve_i(b_2),\\
b_1\ot e_ib_2\quad&\text{if }\vp_i(b_1)<\ve_i(b_2),
\end{array}\right.\\
f_i(b_1\ot b_2)&=\left\{
\begin{array}{ll}
f_ib_1\ot b_2\quad&\text{if }\vp_i(b_1)>\ve_i(b_2),\\
b_1\ot f_ib_2\quad&\text{if }\vp_i(b_1)\le\ve_i(b_2),
\end{array}\right.\\
\wt(b_1\ot b_2)&=\wt b_1+\wt b_2.
\end{align*}
The multiple tensor product is defined inductively. In order to compute the action of $e_i,f_i$ 
on multiple tensor products, it is convenient to use the rule called ``signature rule"~\cite{KN:1994,O}.
Let $b_1\ot b_2\ot\cdots\ot b_m$ be an element of the tensor product of crystals 
$B_1\ot B_2\ot\cdots\ot B_m$. One wishes to find the indices $j,j'$ such that 
\begin{align*}
e_i(b_1\ot\cdots\ot b_m)&=b_1\ot\cdots\ot e_ib_j\ot\cdots\ot b_m,\\
f_i(b_1\ot\cdots\ot b_m)&=b_1\ot\cdots\ot f_ib_{j'}\ot\cdots\ot b_m.
\end{align*}
To do it, we introduce ($i$-)signature by
\[
\overbrace{-\cdots-}^{\ve_i(b_1)}\overbrace{+\cdots+}^{\vp_i(b_1)}
\overbrace{-\cdots-}^{\ve_i(b_2)}\overbrace{+\cdots+}^{\vp_i(b_2)}
\:\cdots\cdots\:
\overbrace{-\cdots-}^{\ve_i(b_m)}\overbrace{+\cdots+}^{\vp_i(b_m)}.
\]
We then reduce the signature by deleting the adjacent $+-$ pair successively. 
Eventually we obtain a reduced signature of the following form.
\[
--\cdots-++\cdots+
\]
Then the action of $e_i$ (resp. $f_i$) corresponds to changing the rightmost $-$ to $+$ 
(resp. leftmost $+$ to $-$). If there is no $-$ (resp. $+$) in the signature, then the action of
$e_i$ (resp. $f_i$) should be set to $\emptyset$. The value of $\ve_i(b)$ (resp. $\vp_i(b)$)
is given by the number of $-$ (resp. $+$) in the reduced signature.

Consider, for instance, an element $b_1\ot b_2\ot b_3$ of the 3 fold tensor product $B_1\ot B_2\ot B_3$. 
Suppose $\ve_i(b_1)=1,\vp_i(b_1)=3,\ve_i(b_2)=1,\vp_i(b_2)=1,
\ve_i(b_3)=2,\vp_i(b_3)=1$. Then the signature and reduced one read 
\[
\begin{array}{cclcccr}
\mbox{sig}&&-++&\cdot&-+&\cdot&--+\phantom{.}\\
\mbox{red sig}&&-&\cdot&&\cdot&+.
\end{array}
\]
Thus we have
\begin{align*}
e_i(b_1\ot b_2\ot b_3)&=e_ib_1\ot b_2 \ot b_3,\\
f_i(b_1\ot b_2\ot b_3)&=b_1\ot b_2 \ot f_ib_3.
\end{align*}

We denote by $B(\La)$ the highest weight crystal of highest weight $\La$, where $\La$
is a dominant integral weight. Let $\La_i$ with $i\in I$ be the fundamental weights associated to a simple
Lie algebra. In this paper, we consider the types of $B_n,C_n$ and $D_n$. 
As usual, a dominant integral weight $\La=\La_{i_1}+\cdots+\La_{i_k}$ is identified with a partition or
Young diagram
with columns of height $i_j$ for $1\le j\le k$, except when $\La_{i_j}$ is a spin weight, namely, 
$\La_n$ for type $B_n$ and $\La_{n-1}$ and $\La_n$ for type $D_n$. 
To represent elements of $B(\La)$ we use Kashiwara--Nakashima (KN) tableaux,
a generalization of semistandard Young tableaux for type $A_n$. For KN tableaux refer to \cite{KN:1994}.
(See also \cite{FOS} for a summary.) Contrary to the original one, we use the French notation where 
parts are drawn in increasing order from top to bottom.

To calculate the actions of $e_i,f_i$ on a KN tableau it is convenient to use so-called the Japanese
reading word of a tableau. For a KN tableau $T$ move from right to left and on each column move from 
bottom to top. During this process we read letters, thereby obtaining a word $w(T)$. A letter can be 
identified with an element of $B(\La_1)$, crystal of the vector representation. Hence $w(T)$ can be
viewed as an element of $B(\La_1)^{\ot N}$ with $N$ being the number of nodes in $T$ or length of $w(T)$.
Then the action of $e_i$ or $f_i$ is calculated by using the signature rule. We still need to remember
the crystal graph of $B(\La_1)$ for type $B_n,C_n,D_n$, but it is easy as described in \cite{KN:1994}.

\subsection{$\pm$-diagrams} \label{subsec:pm diag}

Let $X_n$ be $B_n,C_n$ or $D_n$. For a subset $J \subset I$, we say that $b\in B$ is $J$-highest 
if $e_ib=0$ for all $i\in J$. We set $J=\{2,3,\ldots,n\}$.
We describe $J$-highest elements in terms of a notion of $\pm$-diagram~\cite{S:2008}. 
A $\pm$-diagram $P$ of shape $\La/\la$ is a sequence of partitions $\la\subset \mu \subset \La$ 
such that $\La/\mu$ and $\mu/\la$ are horizontal strips. We
depict this $\pm$-diagram by the skew tableau of shape $\La/\la$ in
which the cells of $\mu/\la$ are filled with the symbol $+$ and
those of $\La/\mu$ are filled with the symbol $-$. Write
$\La=\os(P)$ and $\la=\is(P)$ for the outer and inner shapes of the $\pm$-diagram $P$.
For type $C_n$ we have a further requirement: the outer shape $\La$ contains columns of height 
at most $n$, but the inner shape $\lambda$ is not allowed to be of height $n$ (hence there are 
no empty columns of height $n$). As we have discussed we identify a Young diagram with a weight.

\begin{prop} \cite{S:2008} \label{P:branch}
Let $\La$ be an $X_n$ weight that does not contain spin weights. Then there is an isomorphism of 
$X_{n-1}$-crystals
\begin{align*}
  B_{X_n}(\La) \cong \bigoplus_{\substack{\text{$\pm$-diagrams $P$} \\ \os(P)=\La}}
  B_{X_{n-1}}(\is(P)).
\end{align*}
That is, the multiplicity of $B_{X_{n-1}}(\la)$ in $B_{X_n}(\La)$,
is the number of $\pm$-diagrams of shape $\La/\la$.
\end{prop}

There is a bijection $\Phi:P\mapsto b$ from $\pm$-diagrams $P$ of shape 
$\La/\la$ to the set of $J$-highest elements $b$ of $X_{n-1}$-weight $\la$.
For any columns of height $n$ containing $+$, place a column $12\cdots n$. Otherwise,
place $\ol{1}$ in all positions in $P$ that contain a $-$, and fill the remainder of all 
columns by strings of the form $23\cdots k$. We move through the columns of $b$ from top to bottom,
left to right. Each $+$ in $P$ (starting with the leftmost moving to the right ignoring $+$ at height $n$)
will alter $b$ as we move through the columns. Suppose the $+$ is at height $h$ in $P$.
If one encounters a $\ol{1}$, replace $\ol{1}$ by $\ol{h+1}$. If one encounters a $2$,
replace the string $23\cdots k$ by $12\cdots h h+2\cdots k$.

\begin{example} \label{ex:pm-diag}
Let us consider the following $\pm$-diagram.
\[
\Yvcentermath1
\Yboxdim12pt
\newcommand{\kuu}{{}}
P=\young(+-,\kuu\kuu\kuu+-,\kuu\kuu\kuu\kuu++-,\kuu\kuu\kuu\kuu\kuu\kuu\kuu\kuu+)
\]
To obtain $\Phi(P)$ we first draw the tableau
\[
\Yboxdim12pt
\newcommand{\onebar}{\ol{1}}
\young(5\onebar,4444\onebar,333333\onebar,222222222)\;.
\]
Reading from left there are $+$'s at height 4,3,2,2,1. Each $+$ alter the above tableau as follows.
The 1st $+$ changes the first column as $1234$ (reading from bottom), the 2nd and 3rd change
the second column as $124\ol{4}$, the 4th changes the third column as $124$ and the 5th changes
the fourth column as $134$. Therefore, $\Phi(P)$ is given by
\[
\Yboxdim12pt
\newcommand{\onebar}{\ol{1}}
\newcommand{\fourbar}{\ol{4}}
\young(4\fourbar,3444\onebar,222333\onebar,111122222)\;.
\]
\end{example}

For a word $\boldsymbol{a}=a_1a_2\cdots a_m$ let $e_{\boldsymbol{a}}=e_{a_m}\cdots e_{a_2}e_{a_1}$. We use this
convention also for $f$. Note that the order in $e_{\boldsymbol{a}}$ is reversed from $\boldsymbol{a}$. 
Next proposition shows how we get to the highest element from a $\pm$-diagram by applying $e_i$'s.

\begin{prop}  \label{prop:to highest}
Let $P$ be a $\pm$-diagram whose outer shape has depth $r$. Suppose $r\le n-1$ for $B_n$, $r\le n$
for $C_n$, $r\le n-2$ for $D_n$. Let $c_i$ be the number of columns of
the outer shape with height $i$. Let $c_i^-$ (resp. $c_i^+$) be the number of $-$ (resp. $+$) 
at height $i$. Define a word $\boldsymbol{a}$ by
\[
\boldsymbol{a}=1^{a_1}2^{a_2}\cdots(n-1)^{a_{n-1}}n^{\gamma a_n}(n-1)^{\gamma'a'_{n-1}}(n-2)^{a'_{n-2}}
\cdots1^{a'_1},
\]
where
\[
a_i=\sum_{j=1}^{i-1}c_j^-+(c_i-c_i^+)+\sum_{j=i+1}^n(c_j+c_j^--c_j^+),\qquad
a'_i=\sum_{j=1}^ic_j^-,
\]
where $\gamma=2$ for $B_n$, $\gamma=1$ for the other cases, and $\gamma'=0$ for $D_n$, $\gamma'=1$
for the other cases.
Then $e_{\boldsymbol{a}}\Phi(P)$ is the hightest weight element with highest weight given by its outer shape.
Moreover, at each step when we apply $e_i^{a_i}$ or $e_i^{a'_i}$, including 
$e_n^{\gamma a_n},e_{n-1}^{\gamma'a'_{n-1}}$, the action is maximal, namely, if we apply $e_i^{a_i+1}$
or $e_i^{a'_i+1}$, the outcome turns out $0$.
\end{prop}

\begin{proof}
Suppose $r\neq n$ for type $C_n$.
We first prove the claim when there is no $+$ in $P$. Set $d_i=c_i-c_i^-$. Then the Japanese reading
word of the tableau corresponding to $P$ is given by 
\[
\ol{1}^{c_1^-}2^{d_1}(2\ol{1})^{c_2^-}(23)^{d_2}\cdots(23\cdots r\ol{1})^{c_r^-}(23\cdots r\,r+1)^{d_r}.
\]
The 1-signature is just given by $-^{\ve_1}$, where $\ve_1=c_1+\sum_{i=2}^r(c_i+c_i^+)$, and there is 
no need to reduce. Hence one can apply $e_1^{\ve_1}$. Calculating similarly for $i=2,3,\cdots,n,n-2,
\cdots,1$ one always has a simple $i$-signature of the form $-^{\ve_i}$, and we arrive at the highest
weight element as desired.

Next we consider the general case. We prove by induction on $N$, the number of $+$. If $N=0$, the claim
is proven. Suppose $N>0$ and let $h$ be the height of the lowest $+$ in $P$. Let $P'$ be the same
$\pm$-diagram as $P$ except that there are one less $+$'s at height $h$. Compare the Japanese reading
word of the corresponding tableaux of $P$ and $P'$. The difference is:
\begin{align*}
\text{either (i) }&\text{there is a subword $w=12\cdots h\,h+2\cdots$ in $P$}\\
&\text{but $w'=23\cdots h+1\,h+2\cdots$ in $P'$},\\
\text{or (ii) }&\text{there is a letter $\ol{h+1}$ in $P$ but $\ol{1}$ in $P'$}.
\end{align*}
Apart from this difference in two words, there are subwords of the form $23\cdots$ or letters $\ol{1}$ on 
the left and subwords of the form $12\cdots h'\,h'+2\cdots$ or $\ol{h'+1}$ for some $h'\ge h$ on the
right. Let us calculate the 1-signatures of both words. They are $-^A+^B$ for $P$ and $-^{A+1}+^{B-1}$
for $P'$. (There are no $+-$ pairs.) After applying $e_1^{\max}$ on both $P$ and $P'$, the 2-signatures
also turn out of the form $-^{A'}+^{B'}$ for $P$ and $-^{A'+1}+^{B'-1}$ for $P'$. The difference is 
that there is $12\cdots h\,h+2\cdots$ or $\ol{h+1}$ in $P$ but $13\cdots h+1\,h+2\cdots$ or $\ol{2}$ in
$P'$. Similar situations continue until we apply $e_h$, and after applying $e_h^{\max}$, the two results
coincide. Hence we should have the desired result.

The proof in the case of $r=n$ for type $C_n$ is almost the same. The only difference is that we first
treat the case when there is no $+$ in $P$ except at height $n$, since there is no empty column of 
height $n$. Hence we omit the proof.
\end{proof}

\begin{example}
For a $\pm$-diagram given in Example \ref{ex:pm-diag} set $\boldsymbol{a}=1^72^53^54^35^3\cdots4^33^22$.
Then, according to the previous proposition $e_{\boldsymbol{a}}\Phi(P)$ is a highest weight element.
\end{example}

Later in this paper we will need to apply $e_1$ to a $\pm$-diagram $P$. 
Since $e_1P$ is no longer $J$-highest,
we have to use a pair of $\pm$-diagrams $(P,p)$ to consider $\{3,4,\ldots,n\}$-highest elements.
Namely, $P$ represents a $J$-highest element and $p$ represents a $\{3,4,\ldots,n\}$-highest element
in the $X_{n-1}$-component whose highest weight vector correponds to $P$. Under this bijection we 
identify a $\{3,4,\ldots,n\}$-highest element $b$ with a pair of $\pm$-diagram $(P,p)$. 
To describe the action of $e_1$ on $(P,p)$ perform the following algorithm:
\begin{enumerate}
\item Successively run through all $+$ in $p$ from left to right and, if possible, pair it with 
the leftmost yet unpaired $+$ in $P$ weakly to the left of it.
\item Successively run through all $-$ in $p$ from left to right and, if possible, pair it with
the rightmost yet unpaired $-$ in $P$ weakly to the left.
\item Successively run through all yet unpaired $+$ in $p$ from left to right and, if possible,
pair it with the leftmost yet unpaired $-$ in $p$.
\end{enumerate}

\begin{prop} \cite[Lemma 5.1]{S:2008} \label{prop:e1 action}
If there is an unpaired $+$ in $p$,  $e_1$ moves the rightmost unpaired $+$ in $p$ to $P$. 
Else, if there is an unpaired $-$ in $P$, $e_1$ moves the leftmost unpaired $-$ in $P$ to $p$.
Else $e_1$ annihilates $(P,p)$.
\end{prop}

\section{KR crystal $B^{r,k}$ and combinatorial $R$-matrix}

Let $\geh$ be an affine Lie algebra of type $D_n^{(1)}$, $B_n^{(1)}$, or $A_{2n-1}^{(2)}$ with the 
underlying finite-dimensional simple Lie algebra $\geh_0$ of type $X_n=D_n,B_n$, or $C_n$, respectively.
We label the vertices of the corresponding Dynkin diagram according to \cite{Kac}, so the index set of 
$\geh$ (resp. $\geh_0$) is $I=\{0,1,\ldots,n\}$ (resp. $I_0:=I\setminus\{0\}=\{1,2,\ldots,n\}$).
In this section we review KR crystals $B^{r,k}$ of type $\geh$ given in~\cite{OS:2008,S:2008}
for $k\in\Z_{\ge1}$ and $1\le r\le n-2$ for $D_n^{(1)}$, $1\le r\le n-1$ for $B_n^{(1)}$ and
$1\le r\le n$ for $A_{2n-1}^{(2)}$.

As an $X_n$-crystal, $B^{r,k}$ is given by
\begin{equation} \label{eq:classical decomp}
	B^{r,k} \cong \bigoplus_\La B(\La),
\end{equation}
Here $B(\La)$ is the $X_n$-crystal of highest weight $\La$ and the sum runs over all dominant weights $\La$ 
that can be obtained from $s\La_r$ by the removal of vertical dominoes, where $\La_i$ are the $i$-th 
fundamental weights of $X_n$.

In order to define the actions of $e_0$ and $f_0$ we first consider an automorphism $\sigma$ on the 
KR crystal $B^{r,k}$. The Dynkin diagrams of type $D_n^{(1)},B_n^{(1)}$, and $A_{2n-1}^{(2)}$ all have
an automorphism interchanging nodes $0$ and $1$. $\sigma$ corresponds to this Dynkin diagram automorphism.
By construction $\sigma$ commutes with $e_{i}$ and $f_{i}$ for $i\in J:=\{2,3,\ldots,n\}$. 
Hence it suffices to define $\sigma$ on $J$-highest elements. Because of the
bijection $\Phi$ from $\pm$-diagrams to $J$-highest elements as described in 
Section~\ref{subsec:pm diag}, it suffices to define the corresponding map $\mathfrak{S}$ on $\pm$-diagrams.
Let $P$ be a $\pm$-diagram of shape $\La/\la$. Let $c_i=c_i(\la)$ be
the number of columns of height $i$ in $\la$ for all $1\le i<r$ with
$c_0=k-\la_1$. If $i\equiv r-1 \pmod{2}$, then in $P$, above each
column of $\la$ of height $i$, there must be a $+$ or a $-$.
Interchange the number of such $+$ and $-$ symbols. If $i\equiv r
\pmod{2}$, then in $P$, above each column of $\la$ of height $i$,
either there is no sign or a $\mp$ pair. Suppose there are $p_i$
$\mp$ pairs above the columns of height $i$. Change this to
$(c_i-p_i)$ $\mp$ pairs. The result is $\mathfrak{S}(P)$, which has the
same inner shape $\la$ as $P$ but a possibly different outer shape.

Let $b\in B^{r,k}$ and $\boldsymbol{a}=a_1a_2\cdots a_\ell$ be such that
$e_{\boldsymbol{a}}b$ is a $J$-highest element. Then, $\sigma(b)$ is given by 
\begin{equation} \label{eq:sigma}
\sigma(b) 
= f_{\mathrm{Rev}(\boldsymbol{a})} \circ \Phi \circ \mathfrak{S} \circ \Phi^{-1} \circ e_{\boldsymbol{a}}(b),
\end{equation}
where $\mathrm{Rev}(\boldsymbol{a})=a_\ell\cdots a_2a_1$.
The affine crystal operators $e_{0}$ and $f_{0}$ are then defined as
\begin{equation} \label{eq:0 action}
e_{0} = \sigma \circ e_{1} \circ \sigma, \quad
f_{0} = \sigma \circ f_{1} \circ \sigma.
\end{equation}

If $r=1$, the structure of the KR crystal turns out simple. An crystal element of $B^{1,l}$ 
can be identified with one-row KN tableau of length $l$ with letters from $i,\ol{i}$ ($1\le i\le n$)
(and $0$ for $B_n^{(1)}$). Denoting the number of letters $i,\ol{i}$ or $0$ by $x_i,\ol{x}_i$ or $x_0$, 
we have the so-called coordinate representation of $B^{1,l}$ \cite{KKM,O}.
\begin{equation}
B^{1,l}=\left\{
\begin{array}{l}
\{(x_1,\ldots,x_n,\xb_n,\ldots,\xb_1)\mid x_n\mbox{ or }\xb_n=0,\sum_{i=1}^n(x_i+\xb_i)=l\}
\text{ for }D_n^{(1)},\\
\{(x_1,\ldots,x_n,x_0,\xb_n,\ldots,\xb_1)\mid x_0=0\mbox{ or }1,\sum_{i=1}^n(x_i+\xb_i)+x_0=l\}
\text{ for }B_n^{(1)},\\
\{(x_1,\ldots,x_n,\xb_n,\ldots,\xb_1)\mid \sum_{i=1}^n(x_i+\xb_i)=l\}
\text{ for }A_{2n-1}^{(2)}.
\end{array}
\right.
\end{equation}
The action of $e_i,f_i$ for $i=1,2,\ldots,n$ can be calculated as we explained in the last paragraph
of section \ref{subsec:crystals}. The action of $e_0,f_0$ is given by
\begin{equation}
\begin{split}
e_0b&=\left\{
\begin{array}{l}
(x_1,x_2-1,\ldots,\xb_2,\xb_1+1)\mbox{ if }x_2>\xb_2,\\
(x_1-1,x_2,\ldots,\xb_2+1,\xb_1)\mbox{ if }x_2\le\xb_2,
\end{array}
\right.\\
f_0b&=\left\{
\begin{array}{l}
(x_1,x_2+1,\ldots,\xb_2,\xb_1-1)\mbox{ if }x_2\ge\xb_2,\\
(x_1+1,x_2,\ldots,\xb_2-1,\xb_1)\mbox{ if }x_2<\xb_2.
\end{array}
\right.\\
\end{split}
\end{equation}
We list the values of $\ve_i,\vp_i$ below.
\begin{equation} \label{eq:veps vphi}
\begin{split}
&\ve_0(b)=x_1+(x_2-\xb_2)_+,\quad \vp_0(b)=\xb_1+(\xb_2-x_2)_+,\\
&\ve_i(b)=\xb_i+(x_{i+1}-\xb_{i+1})_+,\quad \vp_i(b)=x_i+(\xb_{i+1}-x_{i+1})_+\text{ if }i\neq0,n,\\
&(\ve_n(b),\vp_n(b))=\left\{
\begin{array}{l}
(\xb_{n-1}+\xb_n,x_{n-1}+x_n)\text{ for }D_n^{(1)},\\
(2\xb_n+x_0,2x_n+x_0)\text{ for }B_n^{(1)},\\
(\xb_n,x_n)\text{ for }A_{2n-1}^{(2)},
\end{array}\right.
\end{split}
\end{equation}
where $(x)_+=\max(x,0)$.

Let us now consider a tensor product of KR crystals $B^{r,k}\ot B^{r',k'}$. It is known
\cite{Ka3,O} that there exists a unique bijection $R$, called combinatorial $R$-matrix, 
commuting with Kashiwara operators $e_i,f_i$ for any $i=0,1,\ldots,n$. Since $R$ preserves the weight, 
$u_0\ot u'_0$ should be sent to $u'_0\ot u_0$ by $R$, where $u_0$ (resp. $u'_0$) is the 
$I_0$-highest elements
of $B(k\La_r)$ (resp. $B(k'\La_{r'})$) in $B^{r,k}$ (resp. $B^{r',k'}$). For the other elements
the image is uniquely determined, since $B^{r,k}\ot B^{r',k'}$ is known to be connected \cite{FOS2}.
Next we explain the energy function $H$. Let $b\ot b'\in B^{r,k}\ot B^{r',k'}$ correspond to 
$\bt'\ot\bt\in B^{r',k'}\ot B^{r,k}$ by $R$. Suppose $e_i(b\ot b')\neq0$. 
Applying $e_i$ on both sides of $R(b\ot b')=\bt'\ot\bt$, we are led to consider the following four cases:
\begin{align*}
\mbox{(LL)}&\quad R(e_ib\ot b')=e_i\bt'\ot\bt,\\
\mbox{(LR)}&\quad R(e_ib\ot b')=\bt'\ot e_i\bt,\\
\mbox{(RL)}&\quad R(b\ot e_ib')=e_i\bt'\ot\bt,\\
\mbox{(RR)}&\quad R(b\ot e_ib')=\bt'\ot e_i\bt.
\end{align*}
Then the function $H$ is uniquely determined, up to adding a constant, by
\begin{equation} \label{eq:e-func}
H(e_i(b\ot b'))=
\begin{cases}
H(b\ot b')+1
&\mbox{ if }i=0\mbox{ and case (LL) occurs},\\
H(b\ot b')-1
&\mbox{ if }i=0\mbox{ and case (RR) occurs},\\
H(b\ot b')&\mbox{ otherwise}.
\end{cases}
\end{equation}
Although it is not obvious that such a function exists, it is shown to exist \cite{Ka3,O}.

Next we investigate conditions for an element of $B^{r,k}\ot B^{1,l}$ or $B^{1,l}\ot B^{r,k}$ to be
$I_0$-highest. Recall the following fundamental fact:
\begin{equation} \label{eq:ht cond}
e_i(b\ot b')=0\mbox{ if and only if }e_ib=0\text{ and }\ve_i(b')\le\langle h_i,\wt b\rangle.
\end{equation}
In particular, if $b\ot b'$ is $I_0$-highest, then $b$ has to be $I_0$-highest.

\begin{prop} \label{prop:ht cond 1}
Let $\mu=\sum_i\mu_i\La_i$ be a dominant integral weight that appears in \eqref{eq:classical decomp} as 
highest weight. By abuse of notation let $\mu$ also stand for the highest KN tableau of weight $\mu$. 
Let $x$ be an element of $B^{1,l}$ represented by coordinates. Then, an element $b\ot x$ of 
$B^{r,k}\ot B^{1,l}$ is $I_0$-highest, if and only if $b=\mu$ for some $\mu$ as above and the following
conditions for $x$ are satisfied.
\begin{itemize}
\item[(i)] $x_i=0$ if $i\ge r+2$,
\item[(ii)] $\ol{x}_i=0$ if $i\ge r+2$ or $i\equiv r+1\;(2)$,
\item[(iii)] $x_{i+1}+\ol{x}_i\le\mu_i$ if $i\le r$ and $i\equiv r\;(2)$,
\item[(iv)] $x_i\le\ol{x}_i$ if $1<i\le r$ and $i\equiv r\;(2)$.
\end{itemize}
In the case of $r=n$ where $\geh=A_{2n-1}^{(2)}$, $x_{n+1}$ appearing in (iii) should be understood as $0$.
\end{prop}

\begin{proof} 
Apply \eqref{eq:ht cond} for $i=1,2,\ldots,n$ and use the formula for $\ve_i$ in \eqref{eq:veps vphi}.
\end{proof}

In what follows, for a $\pm$-diagram $P$ we use the following notation. Let $\ast$ be one of $\cdot,+,-,\mp$
($\cdot$ stands for emptiness). We denote by $p_i^\ast$ the number of columns of the outer shape of $P$ of
height $i$ that contain $\ast$.

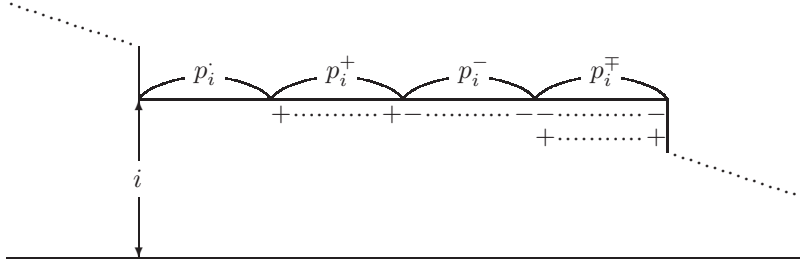
\begin{center}
\unitlength 10pt
\begin{picture}(30,10)
\put(0,0){\line(1,0){30}}
\put(5,6){\line(1,0){20}}
\multiput(0,0)(5,0){4}{
\qbezier(5,6)(5.3,6.5)(6.7,6.8)
\qbezier(8.3,6.8)(9.7,6.5)(10,6)
}
\put(5,6){\line(0,1){2}}
\put(25,6){\line(0,-1){2}}
\multiput(25.3,3.9)(0.3,-0.1){16}{\circle*{0.1}}
\multiput(4.7,8.1)(-0.3,0.1){16}{\circle*{0.1}}
\put(7.1,6.8){$p_i^\cdot$}
\put(12.1,6.8){$p_i^+$}
\put(17.1,6.8){$p_i^-$}
\put(22.1,6.8){$p_i^\mp$}
\put(5,3.7){\vector(0,1){2.3}}
\put(5,2.3){\vector(0,-1){2.3}}
\put(4.8,2.7){$i$}
\put(10,5.2){$+$}
\put(14.2,5.2){$+$}
\put(15,5.2){$-$}
\put(19.2,5.2){$-$}
\put(20,5.2){$-$}
\put(24.2,5.2){$-$}
\put(20,4.3){$+$}
\put(24.2,4.3){$+$}
\multiput(0,0)(5,0){3}{
\multiput(10.95,5.42)(0.3,0){11}{\circle*{0.1}}
}
\multiput(20.95,4.52)(0.3,0){11}{\circle*{0.1}}
\end{picture}
\end{center}

\begin{prop} \label{prop:ht cond 2}
An element $b\ot b'$ of $B^{1,l}\ot B^{r,k}$ is $I_0$-highest, if and only if $b=1^l$ and $b'$ is a 
$J$-highest element whose corresponding $\pm$-diagram $P$ satisfies
\[
\sum_{1\le i\le r\atop i\equiv r\;(2)}(p_i^\cdot+p_i^-)+
\sum_{1<i\le r\atop i\equiv r\;(2)}(p_i^-+p_i^\mp)\le l.
\]
\end{prop}

\begin{proof}
Apply \eqref{eq:ht cond} for $i=1,2,\ldots,n$ and use Proposition \ref{prop:e1 action} to calculate 
$\ve_1$ of the $\pm$-diagram $P$.
\end{proof}

\section{Main result}

We consider the combinatorial $R$-matrix
\[
R:B^{r,k}\otimes B^{1,l}\longrightarrow B^{1,l}\otimes B^{r,k}.
\]
Let $b\ot x\in B^{r,k}\ot B^{1,l}$ be $I_0$-highest and $R(b\ot x)=x'\ot b'$. Then, $x'\ot b'$ is also
$I_0$-highest, and from Propositions \ref{prop:ht cond 1} and \ref{prop:ht cond 2} $b=\mu$ for some 
dominant integral weight $\mu$, $x'=1^l$ and there exists a $\pm$-diagram $P$ such that $b'=\Phi(P)$. 
Thus we have 
\[
R(\mu\ot x)=1^l\ot\Phi(P).
\]
For an element of $B^{1,l}$ we use both the coordinate representation and the Japanese reading word of
the corresponding one-row tableau.
Let $p_i^\ast$ ($\ast=\cdot,+,-,\mp$) be data corresponding to $P$ as in the previous section. Then our
main result is:

\begin{theorem} \label{th:main}
With the notations above we have the following formulas.
\begin{align*}
p_i^\cdot&=\left\{
\begin{array}{ll}
\xb_2-x_2-(x_1-\mu_0)_-&\text{ for }i=0,\\
x_{r+1}&\text{ for }i=r,\\
\xb_{i+2}-x_{i+2}&\text{ otherwise},
\end{array}
\right.\quad
p_i^+=\mu_i-\xb_i-x_{i+1},\\
p_i^-&=\left\{
\begin{array}{ll}
\xb_1&\text{ for }i=1,\\
x_i&\text{ otherwise},
\end{array}
\right.\quad
p_i^\mp=\left\{
\begin{array}{ll}
\min(\mu_0,x_1)&\text{ for }i=2,\\
x_{i-1}&\text{ otherwise}.
\end{array}
\right.
\end{align*}
Here $(x)_-=\min(x,0),\mu_0=k-\sum_{i>0}\mu_i$, and $p_0^\cdot=k-(\text{the first part of the outer shape of }
P)$. We also note $i\equiv r$ (mod 2).

Moreover, the value of the energy function is given by 
\[
H(\mu\ot x)=
\left\{
\begin{array}{ll}
\la_1-k-l&\text{ if $r$ is odd},\\
(\la_1-k)_+-l&\text{ if $r$ is even},
\end{array}
\right.
\]
if we normalize $H$ in such a way as $H((k^r)\ot1^l)=0$. Here $\la_1$ is the first part of the partition
corresponding to the weight of $\mu\ot x$.
\end{theorem}

Solving the formulas for $p_i^\ast$ with respect to $\mu$ and $x$, we obtain

\begin{cor}
The coordinates of the image $\mu\ot x$ of the inverse of $R$ for an element $1^l\ot\Phi(P)$ of 
$B^{1,l}\ot B^{r,k}$ are given by
\begin{align*}
\mu_i&=\left\{
\begin{array}{l}
(l-k+\sum_{i\equiv r\;(2)}(p_i^+-p_i^-))_-+p_4^\mp+p_2^++p_2^-+p_0^\cdot
\text{ if }i=2\text{ and }r\text{ is even},\\
p_{i+2}^\mp+p_i^++p_i^-+p_{i-2}^\cdot
\text{ if }0<i\le r,i\ne2\text{ and }i\equiv r\;(\text{mod }2),
\end{array}
\right.\\
x_i&=\left\{
\begin{array}{l}
(l-k+\sum_{i\equiv r\;(2)}(p_i^+-p_i^-))_++p_2^\mp
\text{ if }i=1\text{ and }r\text{ is even},\\
l-k+\sum_{i\equiv r\;(2)}(p_i^+-p_i^-)+p_1^-
\text{ if }i=1\text{ and }r\text{ is odd},\\
p_i^-\text{ if }i\ne1\text{ and }i\equiv r\;(\text{mod }2),\\
p_{i+1}^\mp\text{ if }i\ne1\text{ and }i\not\equiv r\;(\text{mod }2),
\end{array}
\right.\\
\xb_i&=p_i^-+p_{i-2}^\cdot.
\end{align*}
Here we should understand $p_{r+2}^\mp=p_r^\cdot,p_{-1}^\cdot=0$.
\end{cor}

In what follows in this section we prove Theorem \ref{th:main} by assuming technical propositions in 
later sections. We give a proof
only for type $A_{2n-1}^{(2)}$, since the difference from the other cases is very small as we have seen 
in Proposition \ref{prop:to highest}. Suppose we need to apply $e_{\boldsymbol{a}}$ with such a word 
$\boldsymbol{a}$ as
\[
\boldsymbol{a}=\cdots(r+1)^\alpha(r+2)^\alpha\cdots n^\alpha(n-1)^\alpha\cdots(r+1)^\alpha\cdots
\]
for type $A_{2n-1}^{(2)}$ (see e.g. \eqref{b11}). Then for type $B_n^{(1)}$ we replace it with
\[
\boldsymbol{a}=\cdots(r+1)^\alpha(r+2)^\alpha\cdots n^{2\alpha}(n-1)^\alpha\cdots(r+1)^\alpha\cdots,
\]
and for type $D_n^{(1)}$ 
\[
\boldsymbol{a}=\cdots(r+1)^\alpha(r+2)^\alpha\cdots n^\alpha(n-2)^\alpha\cdots(r+1)^\alpha\cdots.
\]

Consider first the case when $r$ is odd. Suppose $\mu,x$ and $P$ are related as in the statement of the
theorem. We are to show 
\begin{equation} \label{to show}
R(\mu\ot x)=1^l\ot\Phi(P).
\end{equation}
By Proposition \ref{reduction_odd3} showing \eqref{to show} is reduced to the case where $x$ is of the 
form $x=\ol{3}^{\xb_3}1^{l-\xb_3}$. Applying this proposition again to this case, it is then reduced to 
the case where $x=1^l$, since there is no $2$ or $\ol{1}$ in $x$ of the previous case. 
Hence Proposition \ref{prop_special}
completes the proof of \eqref{to show}. (Notice that when $x=1^l,p_i^+=\mu_i$ for any odd $i$ and the 
other $p_i^\ast$ are all zero.) Using these propositions we can calculate $H$ as 
\begin{align*}
H(\mu\ot x)&=H(\ol{\mu}\ot\ol{3}^{x_2+\xb_1}1^{l-x_2-\xb_1})+(x_1+x_2-l)\\
&=H(\ol{\mu}\ot1^l)+(-x_2-\xb_1)+(x_1+x_2-l)\\
&=x_1-\xb_1-l=\la_1-k-l,
\end{align*}
since $\la_1=k+x_1-\xb_1$.

The case when $r$ is even can be proven similarly by using Propositions \ref{reduction_even3} and
\ref{prop_special}.

\section{Proof of a special case}

\subsection{Statement}

Let $\mu=\sum_{i=1}^mc_i\La_{j_i}$ be a dominant integral weight whose corresponding Young diagram is 
depicted as follows.
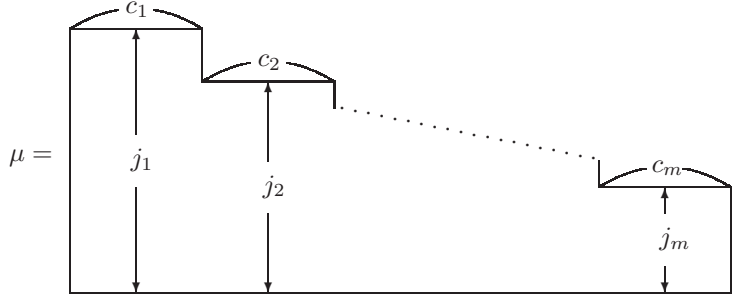
\begin{center}
\unitlength 10pt
\begin{picture}(25,11.5)(-1,0)
\put(-2.3,5){$\mu =$}
%\thicklines
\put(0,0){\line(1,0){25}}
\put(0,0){\line(0,1){10}}
\put(0,10){\line(1,0){5}}
\put(5,10){\line(0,-1){2}}
\put(5,8){\line(1,0){5}}
\put(10,8){\line(0,-1){1}}
\multiput(10.3,7)(0.5,-0.1){20}{\circle*{0.1}}
\put(20,4){\line(0,1){1}}
\put(25,4){\line(-1,0){5}}
\put(25,0){\line(0,1){4}}
%\thinlines
\put(2.3,4.8){$j_1$}
\put(2.5,6){\vector(0,1){4}}
\put(2.5,4){\vector(0,-1){4}}
\put(7.3,3.8){$j_2$}
\put(7.5,5){\vector(0,1){3}}
\put(7.5,3){\vector(0,-1){3}}
\put(22.3,1.8){$j_m$}
\put(22.5,3){\vector(0,1){1}}
\put(22.5,1){\vector(0,-1){1}}
%%
%\thicklines
\qbezier(0,10)(1,10.6)(1.8,10.7)
\qbezier(5,10)(4,10.6)(3.2,10.7)
\put(2.1,10.6){$c_1$}
\qbezier(5,8)(6,8.6)(6.8,8.7)
\qbezier(10,8)(9,8.6)(8.2,8.7)
\put(7.1,8.6){$c_2$}
\qbezier(20,4)(21,4.6)(21.8,4.7)
\qbezier(25,4)(24,4.6)(23.2,4.7)
\put(22.0,4.6){$c_m$}
\end{picture}
\end{center}
Here one can assume
$j_{i-1}-j_{i}\in 2\mathbb{Z}_{>0}$
and $c_i>0$.
We also assume that
\begin{equation}\label{s:eq:sumc_i=k}
\sum_{i=1}^mc_i=k,\qquad j_m>0.
\end{equation}
Then the claim of this section is
the following special version of the main theorem:

\begin{prop}\label{prop_special}
Let $\mu$ be as above. Then we have
\begin{align}
R(\mu\otimes 1^l)=
1^l\otimes\mu ,\qquad
H(\mu\otimes 1^l)=0.
\nonumber
\end{align}
\end{prop}

\begin{remark}
The condition (\ref{s:eq:sumc_i=k})
for $\mu\in B^{r,k}$ is necessary.
For example, in type $D^{(1)}_4$,
the image of the combinatorial $R$-matrix and the value of the energy function for 
$\Yvcentermath1\young(2,1)\otimes\young(1)
\in B^{2,2}\otimes B^{1,1}$ are
$
\Yvcentermath1
\Yboxdim12pt
\newcommand{\bartwo}{\overline{2}}
\young(1)\otimes\young(2\bartwo,12)\,$ and $-1$.
\end{remark}

We divide the proof of this proposition into three parts.
Let us define two words that will be used in the proof.
\begin{align}
\boldsymbol{a}_1=&\,
\underbrace{j_m^{c_m}(j_m+1)^{c_m}}
\underbrace{(j_m-1)^{c_m}j_m^{c_m}}
\underbrace{(j_m-2)^{c_m}(j_m-1)^{c_m}}
\cdots
\underbrace{1^{c_m}2^{c_m}\mathstrut},
\nonumber\\
\boldsymbol{a}_2=&\,
\boldsymbol{a}_{2,1}
\boldsymbol{a}_{2,2}
\boldsymbol{a}_{2,3}
\boldsymbol{a}_{2,4},
\nonumber
\end{align}
where underbraces are introduced to show a unit of repetitions, and 
$\boldsymbol{a}_{2,1},\cdots,\boldsymbol{a}_{2,4}$
are defined as follows.
Set $\delta=(l-c_m)_+$, then
$
\boldsymbol{a}_{2,1}=
\boldsymbol{a}_{2,1,m}
\boldsymbol{a}_{2,1,m-1}
\boldsymbol{a}_{2,1,m-2}
\cdots
\boldsymbol{a}_{2,1,2}
\nonumber
$
where
\begin{align}
&\boldsymbol{a}_{2,1,m}=
\underbrace{2^{2k+\delta}1^{2k}\mathstrut}
\underbrace{3^{2k+\delta}2^{2k}\mathstrut}
\cdots
\underbrace{(j_m+1)^{2k+\delta}j_m^{2k}},
\nonumber\\
&\boldsymbol{a}_{2,1,m-1}=
\underbrace{(j_m+2)^{2k-c_m+\delta}(j_m+1)^{2k-c_m}}
\underbrace{(j_m+3)^{2k-c_m+\delta}(j_m+2)^{2k-c_m}}
\cdots
\nonumber\\
&\hspace{60mm}
\cdots
\underbrace{(j_{m-1}-1)^{2k-c_m+\delta}(j_{m-1}-2)^{2k-c_m}},
\nonumber
\end{align}
and for $m-2\geq i\geq 2$,
\begin{align}
\boldsymbol{a}_{2,1,i}=&\,
\underbrace{j_{i}^{2k-\sum_{j=i}^{m}c_j+\delta}
(j_{i}-1)^{2k-\sum_{j=i}^{m}c_j}}
\underbrace{(j_{i}+1)^{2k-\sum_{j=i}^{m}c_j+\delta}
j_{i}^{2k-\sum_{j=i}^{m}c_j}}
\cdots
\nonumber\\
&\hspace{40mm}
\cdots
\underbrace{(j_{i-1}-1)^{2k-\sum_{j=i}^{m}c_j+\delta}
(j_{i-1}-2)^{2k-\sum_{j=i}^{m}c_j}}.
\nonumber
\end{align}
$\boldsymbol{a}_{2,2}$ is defined by
\begin{align}
\boldsymbol{a}_{2,2}=&\,
{j_1}^{k+\delta}(j_1+1)^{k+\delta}\cdots n^{k+\delta}(n-1)^{k+\delta}\cdots{j_1}^{k+\delta}
\nonumber\\
&\,
\cdot(j_1-1)^{2k+\delta}\cdot {j_1}^k(j_1+1)^k\cdots n^k(n-1)^k\cdots{j_1}^k
\nonumber
\end{align}
and $\boldsymbol{a}_{2,3}=
\boldsymbol{a}_{2,3,1}
\cdots
\boldsymbol{a}_{2,3,m-2}
\boldsymbol{a}_{2,3,m-1}$
where for $1\leq i\leq m-2$,
\begin{align*}
\boldsymbol{a}_{2,3,i}=&\,
\underbrace{(j_i-2)^{k-\sum_{j=1}^ic_j+\delta}
(j_i-1)^{k-\sum_{j=1}^ic_j}}
\underbrace{(j_i-3)^{k-\sum_{j=1}^ic_j+\delta}
(j_i-2)^{k-\sum_{j=1}^ic_j}}\cdots\\
&\hspace{60mm}
\cdots
\underbrace{(j_{i+1}-1)^{k-\sum_{j=1}^ic_j+\delta}
j_{i+1}^{k-\sum_{j=1}^ic_j}}
\end{align*}
and $\boldsymbol{a}_{2,3,m-1}$ is
\begin{align*}
&\underbrace{(j_{m-1}-2)^{k-\sum_{j=1}^{m-1}c_j+\delta}
(j_{m-1}-1)^{k-\sum_{j=1}^{m-1}c_j}}
\underbrace{(j_{m-1}-3)^{k-\sum_{j=1}^{m-1}c_j+\delta}
(j_{m-1}-2)^{k-\sum_{j=1}^{m-1}c_j}}
\\
&\hspace{50mm}\cdots\cdots
\underbrace{(j_{m}+1)^{k-\sum_{j=1}^{m-1}c_j+\delta}
(j_{m}+2)^{k-\sum_{j=1}^{m-1}c_j}}.
\end{align*}
Finally,
$
\boldsymbol{a}_{2,4}=
j_m^\delta
(j_m-1)^\delta\cdots
2^\delta
1^\delta.
$

In the process of proof, we use a dominant integral weight
$\bar{\mu}$ given by $\bar{\mu}=\sum_{i=1}^{m-1}c_i\La_{j_i}+c_m\La_{j_m+2}$.
We assume that $c_m$ is even.
The proof for odd $c_m$ is similar.
During the proof, we often identify a KN tableau with its 
Japanese reading word.

\subsection{Proof: Part 1}
The goal of this subsection is the following lemma:

\begin{lemma}\label{lem_special1}
$
e_{\boldsymbol{a}_2}
e_0^{2k-c_m+\delta}
f_{\boldsymbol{a}_1}
(\mu\otimes 1^l)
=\bar{\mu}\otimes 1^l.
$
\end{lemma}

This is a direct consequence of the following three sublemmas.

\begin{lemma}
\begin{align}
f_{\boldsymbol{a}_1}(\mu\otimes 1^l)=
(34\cdots (j_{m}+2))^{c_m}
(12\cdots j_{m-1})^{c_{m-1}}
\cdots
(12\cdots j_{2})^{c_{2}}
(12\cdots j_1)^{c_1}
\otimes 1^l.
\nonumber
\end{align}
\end{lemma}

\begin{lemma}\label{lem:special1}
\begin{align}
e_0^{2k-c_m+\delta}
f_{\boldsymbol{a}_1}(\mu\otimes 1^l)
=&(34\cdots (j_m+2)\overline{2}\,\overline{1})^{c_m}
(34\cdots j_{m-1}\overline{2}\,\overline{1})^{c_{m-1}}
\cdots
\nonumber\\
&\cdots
(34\cdots j_2\overline{2}\,\overline{1})^{c_2}
(34\cdots j_1\overline{2}\,\overline{1})^{c_1}
\otimes
\overline{2}\mathstrut^{\delta}
1^{l-\delta}.
\nonumber
\end{align}
\end{lemma}

\begin{proof}
To begin with we apply $e_0$ on $f_{\boldsymbol{a}_1}(\mu)$ for maximal times.
Define a word $\boldsymbol{a}_3$ by
$\boldsymbol{a}_3=2^{c_m}3^{c_m}\cdots (j_m+1)^{c_m}.$
Then the $J$-highest element of
$f_{\boldsymbol{a}_1}(\mu)$ is
$$e_{\boldsymbol{a}_3}(f_{\boldsymbol{a}_1}(\mu))=
(23\cdots (j_m+1))^{c_m}
(12\cdots j_{m-1})^{c_{m-1}}
\cdots
(12\cdots j_{2})^{c_{2}}
(12\cdots j_1)^{c_1}.$$
By the map $\Phi^{-1}$ we get the
corresponding $\pm$-diagram and $\mathfrak{S}$ acts on it as follows:
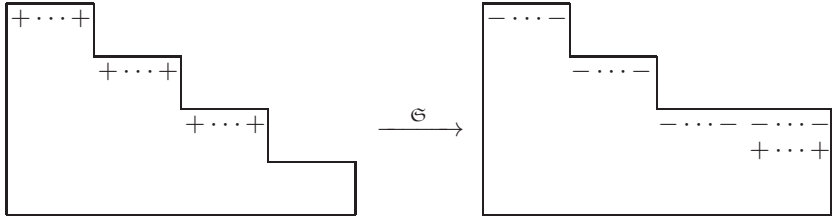
\begin{center}
\unitlength 10pt
\begin{picture}(31.2,8)
\multiput(0,0)(18,0){2}{
\put(0,0){\line(1,0){13.2}}
\put(0,0){\line(0,1){8}}
\put(0,8){\line(1,0){3.3}}
\put(3.3,8){\line(0,-1){2}}
\put(3.3,6){\line(1,0){3.3}}
\put(6.6,6){\line(0,-1){2}}
\put(6.6,4){\line(1,0){3.3}}
\put(13.2,2){\line(0,-1){2}}
}
\put(9.9,4){\line(0,-1){2}}
\put(9.9,2){\line(1,0){3.3}}
\put(27.9,4){\line(1,0){3.3}}
\put(31.2,4){\line(0,-1){2}}
\put(0.15,7.2){$+\cdots +$}
\put(3.45,5.2){$+\cdots +$}
\put(6.75,3.2){$+\cdots +$}
\put(14,3){$\xrightarrow{\quad\mathfrak{S}\quad}$}
\put(18.15,7.2){$-\cdots -$}
\put(21.35,5.2){$-\cdots -$}
\put(24.65,3.2){$-\cdots -$}
\put(28.1,3.2){$-\cdots -$}
\put(28.1,2.2){$+\cdots +$}
\end{picture}
\end{center}
Note that there are $c_m$ $+$'s at height $(j_m+1)$
of the right $\pm$-diagram.
Assume that $c_m$ satisfies
$\sum_{j=1}^{i-1}2c_j<c_m\leq
\sum_{j=1}^{i}2c_j$.
Then $\Phi\circ\mathfrak{S}\circ\Phi^{-1}
\circ e_{\boldsymbol{a}_3}\circ f_{\boldsymbol{a}_1}(\mu)$ is
\begin{align}
&(23\cdots (j_m+2)\overline{1})^{c_{m}}
\label{eq:w_3}\\
&(23\cdots (j_m+2)(j_m+3)(j_m+4)\cdots
j_{m-1}\overline{1})^{c_{m-1}}
\cdots\cdots
\nonumber\\
&(23\cdots (j_m+2)(j_m+3)(j_m+4)\cdots
j_{i-1}\overline{1})^{c_{i-1}}
\nonumber\\
&(23\cdots (j_m+2)(j_m+3)(j_m+4)\cdots
j_i\overline{1})^{\sum_{j=1}^ic_j-c_m/2}
\nonumber\\
&(12\cdots (j_m+1)(j_m+3)(j_m+4)\cdots
j_i\overline{(j_m+2)})^{c_m/2-\sum_{j=1}^{i-1}c_j}
\nonumber\\
&(12\cdots (j_m+1)(j_m+3)(j_m+4)\cdots
j_{i-1}\overline{(j_m+2)})^{c_{i-1}}\cdots\cdots
\nonumber\\
&(12\cdots (j_m+1)(j_m+3)(j_m+4)\cdots
j_1\overline{(j_m+2)})^{c_1}
\nonumber
\end{align}
and $f_{\mathrm{Rev}(\boldsymbol{a}_3)}\circ
\Phi\circ\mathfrak{S}\circ\Phi^{-1}
\circ e_{\boldsymbol{a}_3}\circ f_{\boldsymbol{a}_1}(\mu)$ is
\begin{align}%\label{eq:w_3}
&
(234\cdots (j_m+2)\overline{1})^{c_m}
(234\cdots j_{m-1}\overline{1})^{c_{m-1}}
\cdots
(234\cdots j_{i-1}\overline{1})^{c_i}
\nonumber\\
&
(234\cdots j_i\overline{1})^{\sum_{j=1}^ic_j-c_m/2}
(134\cdots j_i\overline{2})^{c_m/2-\sum_{j=1}^{i-1}c_j}
\nonumber\\
&
(134\cdots j_{i+1}\overline{2})^{c_{i+1}}
\cdots
(134\cdots j_2\overline{2})^{c_2}
(134\cdots j_1\overline{2})^{c_1}.
\nonumber
\end{align}
{}From this expression, we get $\varepsilon_0(f_{\boldsymbol{a}_1}(\mu))=2k-c_m,
\varphi_0(f_{\boldsymbol{a}_1}(\mu))=c_m$.
Applying
$e_1^{\varepsilon_0(f_{\boldsymbol{a}_1}(\mu))}$
we get
\begin{align}
w_1:=&\, e_1^{2k-c_m}\circ
f_{\mathrm{Rev}(\boldsymbol{a}_3)}\circ
\Phi\circ\mathfrak{S}\circ\Phi^{-1}
\circ e_{\boldsymbol{a}_3}\circ f_{\boldsymbol{a}_1}(\mu)
\nonumber\\
=&\, (134\cdots (j_m+2)\overline{2})^{c_m}
(134\cdots j_{m-1}\overline{2})^{c_{m-1}}
\cdots
(134\cdots j_2\overline{2})^{c_2}
(134\cdots j_1\overline{2})^{c_1}.
\nonumber
\end{align}

To convert the action of $e_1^{2k-c_m}$ into that of $e_0^{2k-c_m}$,
we need to define the words
$\boldsymbol{a}_{4}=\boldsymbol{a}_{4,1}
\boldsymbol{a}_{4,2}\boldsymbol{a}_{4,3}$
as follows.
$
\boldsymbol{a}_{4,1}=
\boldsymbol{a}_{4,1,m+1}
\boldsymbol{a}_{4,1,m}
\boldsymbol{a}_{4,1,m-1}\cdots
\boldsymbol{a}_{4,1,2}
$
where the subwords $\boldsymbol{a}_{4,1,i}$ are
\begin{align*}
&\boldsymbol{a}_{4,1,m+1}=
2^{2k}3^{2k}\cdots (j_m+1)^{2k},\\
&\boldsymbol{a}_{4,1,m}=
(j_m+2)^{2k-c_m}(j_m+3)^{2k-c_m}\cdots
(j_{m-1}-1)^{2k-c_m},\\
&\boldsymbol{a}_{4,1,i}=
j_{i}^{2k-\sum_{j=i}^mc_j}
(j_{i}+1)^{2k-\sum_{j=i}^mc_j}\cdots
(j_{i-1}-1)^{2k-\sum_{j=i}^mc_j}
\,\, (m-1\geq i\geq 2).
\end{align*}
Define $\boldsymbol{a}_{4,2}=
j_1^k (j_1+1)^k\cdots
n^k(n-1)^k\cdots
(j_1+1)^k j_1^k$
and $\boldsymbol{a}_{4,3}=
\boldsymbol{a}_{4,3,1}
\cdots
\boldsymbol{a}_{4,3,m-1}$
where the subwords $\boldsymbol{a}_{4,3,i}$ are
\begin{align*}
&\boldsymbol{a}_{4,3,i}=
(j_i-1)^{k-\sum_{j=1}^ic_j}
(j_i-2)^{k-\sum_{j=1}^ic_j}\cdots
j_{i+1}^{k-\sum_{j=1}^ic_j},
\quad (2\leq i\leq m-2),\\
&\boldsymbol{a}_{4,3,m-1}=
(j_{m-1}-1)^{k-\sum_{j=1}^{m-1}c_j}
(j_{m-1}-2)^{k-\sum_{j=1}^{m-1}c_j}\cdots
(j_{m}+2)^{k-\sum_{j=1}^{m-1}c_j}.
\end{align*}
Then, starting from $e_{\boldsymbol{a}_{4,1}}(w_1)$,
one calculates
\begin{align*}
&
(123\cdots (j_m+1)\overline{j_1})^{c_{m}}
(123\cdots (j_{m-1}-1)\overline{j_{1}})^{c_{m-1}}
\cdots
(123\cdots (j_1-1)\overline{j_1})^{c_1}\\
\xrightarrow{e_{\boldsymbol{a}_{4,2}}}
&
(123\cdots (j_m+1){j_1})^{c_{m}}
(123\cdots (j_{m-1}-1){j_{1}})^{c_{m-1}}
\cdots
(123\cdots (j_1-1){j_1})^{c_1}\\
\xrightarrow{e_{\boldsymbol{a}_{4,3}}}
&
(12\cdots (j_m+2))^{c_m}
(12\cdots j_{m-1})^{c_{m-1}}
\cdots
(12\cdots j_1)^{c_1}.
\end{align*}
Here $b\xrightarrow{e_{\boldsymbol{a}}}b'$ means $e_{\boldsymbol{a}}b=b'$.
With $\Phi^{-1}$, this corresponds to the following $\pm$-diagram and 
$\mathfrak{S}$ acts on it as follows:
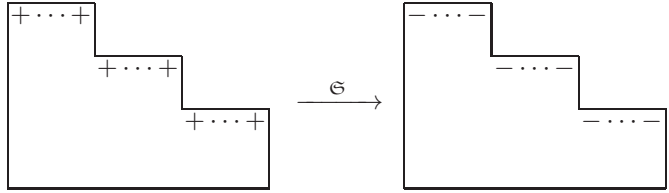
\begin{center}
\unitlength 10pt
\begin{picture}(25,8)
\multiput(0,0)(15,0){2}{
\put(0,0){\line(1,0){9.9}}
\put(0,0){\line(0,1){7}}
\put(0,7){\line(1,0){3.3}}
\put(3.3,7){\line(0,-1){2}}
\put(3.3,5){\line(1,0){3.3}}
\put(6.6,5){\line(0,-1){2}}
\put(6.6,3){\line(1,0){3.3}}
\put(9.9,3){\line(0,-1){3}}
}
\put(0.1,6.3){$+\cdots +$}
\put(3.4,4.3){$+\cdots +$}
\put(6.7,2.3){$+\cdots +$}
\put(10.9,3){$\xrightarrow{\quad\mathfrak{S}\quad}$}
\put(15.1,6.3){$-\cdots -$}
\put(18.4,4.3){$-\cdots -$}
\put(21.7,2.3){$-\cdots -$}
\end{picture}
\end{center}
Thus, starting from
$\Phi\circ\mathfrak{S}\circ\Phi^{-1}\circ
e_{\boldsymbol{a}_{4}}(w_1)$, we calculate
\begin{align*}
&
(23\cdots (j_m+2)\overline{1})^{c_m}
(23\cdots j_{m-1}\overline{1})^{c_{m-1}}
\cdots
(23\cdots j_1\overline{1})^{c_1}
\\
\xrightarrow{f_{\mathrm{Rev}(\boldsymbol{a}_{4,3})}}
&
(23\cdots (j_m+1)j_1\overline{1})^{c_m}
(23\cdots (j_{m-1}-1)j_1\overline{1})^{c_{m-1}}
\cdots
\\
&\hspace{50mm}\cdots
(23\cdots (j_2-1)j_1\overline{1})^{c_2}
(23\cdots (j_1-1)j_1\overline{1})^{c_1}
\\
\xrightarrow{f_{\mathrm{Rev}(\boldsymbol{a}_{4,2})}}
&
(23\cdots (j_m+1)\overline{j_1}\,\overline{1})^{c_m}
(23\cdots (j_{m-1}-1)\overline{j_1}\,\overline{1})^{c_{m-1}}
\cdots\\
&\hspace{50mm}\cdots
(23\cdots (j_2-1)\overline{j_1}\,\overline{1})^{c_2}
(23\cdots (j_1-1)\overline{j_1}\,\overline{1})^{c_1}
\\
\xrightarrow{f_{\mathrm{Rev}(\boldsymbol{a}_{4,1})}}
&
(34\cdots (j_m+2)\overline{2}\,\overline{1})^{c_m}
(34\cdots j_{m-1}\overline{2}\,\overline{1})^{c_{m-1}}
\cdots
(34\cdots j_2\overline{2}\,\overline{1})^{c_2}
(34\cdots j_1\overline{2}\,\overline{1})^{c_1},
\end{align*}
where the final formula gives
$e_0^{2k-c_m}(f_{\boldsymbol{a}_{1}}(\mu))$.

{}From
$\varepsilon_0(f_{\boldsymbol{a}_{1}}(\mu))$ and
$\varphi_0(f_{\boldsymbol{a}_{1}}(\mu))$,
we see that the 0-signature of
$f_{\boldsymbol{a}_1}(\mu\otimes 1^l)$
is $-^{2k-c_m}+^{c_m}-^{l}$.
Therefore we get
\begin{align}
e_0^{2k-c_m+(l-c_m)_+}
f_{\boldsymbol{a}_1}(\mu\otimes 1^l)
=&(34\cdots (j_m+2)\overline{2}\,\overline{1})^{c_m}
(34\cdots j_{m-1}\overline{2}\,\overline{1})^{c_{m-1}}
\cdots
\nonumber\\
&\cdots
(34\cdots j_2\overline{2}\,\overline{1})^{c_2}
(34\cdots j_1\overline{2}\,\overline{1})^{c_1}
\otimes
\overline{2}\mathstrut^{(l-c_m)_+}1^{l-(l-c_m)_+},
\nonumber
\end{align}
which gives the desired expression.
\end{proof}

\begin{lemma}\label{lem:special2}
Starting from
$e_{\boldsymbol{a}_{2,1}}
e_0^{2k-c_m+\delta}
f_{\boldsymbol{a}_1}(\mu\otimes 1^l)$, we have
\begin{align*}
&
(123\cdots j_m\overline{j_1}\,\overline{(j_1-1)})^{c_m}
(123\cdots (j_{m-1}-2)\overline{j_1}\,\overline{(j_1-1)})^{c_{m-1}}
\cdots
\\
&
(123\cdots (j_2-2)\overline{j_1}\,\overline{(j_1-1)})^{c_2}
(123\cdots (j_1-2)\overline{j_1}\,\overline{(j_1-1)})^{c_1}
\otimes
\overline{j_1}\mathstrut^{\delta}
1^{l-\delta}
\\
\xrightarrow{e_{\boldsymbol{a}_{2,2}}}
&
(123\cdots j_m(j_1-1)j_1)^{c_m}
(123\cdots (j_{m-1}-2)(j_1-1)j_1)^{c_{m-1}}
\cdots
\\
&
(123\cdots (j_2-2)(j_1-1)j_1)^{c_2}
(123\cdots (j_1-2)(j_1-1)j_1)^{c_1}
\otimes
(j_1-1)^{\delta}
1^{l-\delta}
\\
\xrightarrow{e_{\boldsymbol{a}_{2,3}}}
&
(123\cdots j_m(j_m+1)(j_m+2))^{c_m}
(123\cdots (j_{m-1}-2)(j_{m-1}-1)j_{m-1})^{c_{m-1}}
\cdots
\\
&
(123\cdots (j_2-2)(j_2-1)j_2)^{c_2}
(123\cdots (j_1-2)(j_1-1)j_1)^{c_1}
\otimes
(j_m+1)^{\delta}
1^{l-\delta}
\\
\xrightarrow{e_{\boldsymbol{a}_{2,4}}}
&
(123\cdots j_m(j_m+1)(j_m+2))^{c_m}
(123\cdots (j_{m-1}-2)(j_{m-1}-1)j_{m-1})^{c_{m-1}}
\cdots
\\
&
(123\cdots (j_2-2)(j_2-1)j_2)^{c_2}
(123\cdots (j_1-2)(j_1-1)j_1)^{c_1}
\otimes
1^l,
\end{align*}
where the final expression is equal to $\bar{\mu}\otimes 1^l$.
\end{lemma}

Since $e_{\boldsymbol{a}_{2,4}}
e_{\boldsymbol{a}_{2,3}}
e_{\boldsymbol{a}_{2,2}}
e_{\boldsymbol{a}_{2,1}}=e_{\boldsymbol{a}_{2}}$,
we have finished the proof of Lemma \ref{lem_special1}.

\subsection{Proof: Part 2}
The goal of this subsection is the following lemma:
\begin{lemma}\label{lem_special2}
$e_{\boldsymbol{a}_2}
e_0^{2k-c_m+\delta}
f_{\boldsymbol{a}_1}
(1^l\otimes\mu )
=1^l\otimes\bar{\mu}.$
\end{lemma}
To begin with, we have
\begin{align}
&f_{\boldsymbol{a}_1}(1^l\otimes
(123\cdots j_m)^{c_m}
(123\cdots j_{m-1})^{c_{m-1}}
\cdots
(123\cdots j_1)^{c_1}
)\nonumber\\
=\,&
f_2^{c_m}f_{1}^{c_m}
(1^l\otimes
(145\cdots (j_m+2))^{c_m}
(123\cdots j_{m-1})^{c_{m-1}}
\cdots
(123\cdots j_1)^{c_1}
).
\nonumber
\end{align}
Here, we need to divide the calculation into two cases. \bigskip\\
{\it Case 1:}
If $l<c_m$, we have
\begin{align*}
f_{\boldsymbol{a}_1}(1^l\otimes\mu)=&\,
3^l\otimes
(345\cdots (j_m+2))^{c_m-l}
(145\cdots (j_m+2))^l
\\
&\hspace{8mm}
(123\cdots j_{m-1})^{c_{m-1}}
\cdots
(123\cdots j_1)^{c_1}.
\\
=:&\, 3^l\otimes w_2.
\end{align*}
{\it Case 2:}
If $l\geq c_m$, we have
\begin{align}
f_{\boldsymbol{a}_1}(1^l\otimes\mu)=
3^{c_m}1^{l-c_m}
\otimes
(145\cdots (j_m+2))^{c_m}
(123\cdots j_{m-1})^{c_{m-1}}
\cdots
(123\cdots j_1)^{c_1}.
\nonumber
\end{align}

\subsubsection{Proof for Case 1.}
To begin with we remark that in this case
we have
\begin{align}
2k-c_m=2k-c_m+\delta =2k-c_m+(l-c_m)_+.
\nonumber
\end{align}
\begin{lemma}
Assume that $l$ satisfies
$\sum_{j=1}^{i-1}2c_j<l\leq
\sum_{j=1}^i2c_j$.
Then
\begin{align*}
&e_0^{2k-c_m}
f_{\boldsymbol{a}_1}(1^l\otimes\mu)=
\\
&\hspace{5mm}
3^l\otimes
(345\cdots (j_m+2)\overline{2}\,\overline{1})^{c_m}
(345\cdots j_{m-1}\overline{2}\,\overline{1})^{c_{m-1}}
\cdots
(345\cdots j_{i-1}\overline{2}\,\overline{1})^{c_{i-1}}
\\
&\hspace{12mm}
(345\cdots j_i\overline{2}\,\overline{1})^{\sum_{j=1}^ic_j-l/2}
(145\cdots j_i\overline{3}\,\overline{2})^{l/2-\sum_{j=1}^{i-1}c_j}
\\
&\hspace{12mm}
(145\cdots j_{i+1}\overline{3}\,\overline{2})^{c_{i+1}}
\cdots
(145\cdots j_2\overline{3}\,\overline{2})^{c_2}
(145\cdots j_1\overline{3}\,\overline{2})^{c_1}.
\end{align*}
\end{lemma}
\begin{proof}
Define a word $\boldsymbol{a}_5$ by
$\boldsymbol{a}_5=
2^{c_m-l}
\underbrace{3^{c_m}2^l\mathstrut}
\underbrace{4^{c_m}3^l\mathstrut}
\cdots
\underbrace{(j_m+1)^{c_m}j_m^l}$.
Then we have
\begin{align*}
e_{\boldsymbol{a}_5}(w_2)=
(234\cdots (j_m+1))^{c_m-l}
(123\cdots j_m)^l
(123\cdots j_{m-1})^{c_{m-1}}
\cdots
(123\cdots j_1)^{c_1}.
\end{align*}
By the map $\Phi^{-1}$,
$e_{\boldsymbol{a}_5}(w_2)$
corresponds to the following $\pm$-diagram, $\mathfrak{S}$ acts on it as follows:
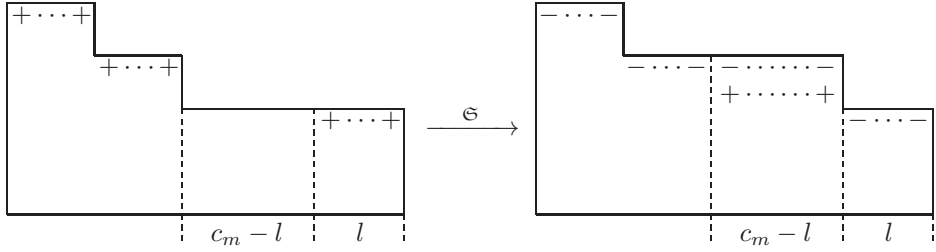
\begin{center}
\unitlength 10pt
\begin{picture}(35,10)
\multiput(0,1)(20,0){2}{
\put(0,0){\line(1,0){15}}
\put(0,0){\line(0,1){8}}
\put(0,8){\line(1,0){3.3}}
\put(3.3,8){\line(0,-1){2}}
\put(3.3,6){\line(1,0){3.3}}
%
%\put(11.6,4){\line(0,-1){1}}
\put(11.6,4){\line(1,0){3.4}}
\put(15,4){\line(0,-1){4}}
\put(7.7,-1){$c_m-l$}
\put(13.2,-1){$l$}
\multiput(6.6,-1)(0,0.4){18}{\line(0,1){0.2}}
\multiput(11.6,-1)(0,0.4){13}{\line(0,1){0.2}}
\multiput(15,-1)(0,0.4){5}{\line(0,1){0.2}}
}
\put(6.6,7){\line(0,-1){2}}
\put(6.6,5){\line(1,0){5}}
\put(26.6,7){\line(1,0){5}}
\put(31.6,7){\line(0,-1){2}}
\put(0.15,8.3){$+\cdots +$}
\put(3.45,6.3){$+\cdots +$}
\put(11.85,4.3){$+\cdots +$}
\put(15.8,4){$\xrightarrow{\hspace{4mm}\mathfrak{S}\hspace{4mm}}$}
\put(20.1,8.3){$-\cdots -$}
\put(23.4,6.3){$-\cdots -$}
\put(26.95,5.3){$+\cdots\cdots +$}
\put(26.95,6.3){$-\cdots\cdots -$}
\put(31.8,4.3){$-\cdots -$}
\end{picture}
\end{center}
There are $(c_m-l)$ $+$'s at height $(j_m+1)$ in the right $\pm$-diagram.
Assume that $(c_m-l)$ satisfies
$\sum_{j=1}^{i-1}2c_j<(c_m-l)\leq
\sum_{j=1}^i2c_j$.
We also assume that this $i$ satisfies
$i<m$ for the sake of
simplicity.
Then $\Phi\circ\mathfrak{S}\circ
\Phi^{-1}\circ
e_{\boldsymbol{a}_5}(w_2)$ is
\begin{align*}
&
(23\cdots j_m\overline{1})^l
(23\cdots (j_m+2)\overline{1})^{c_m-l}
\\
&
(23\cdots (j_m+2)(j_m+3)(j_m+4)\cdots
j_{m-1}\overline{1})^{c_{m-1}}
\cdots\cdots
\\
&
(23\cdots (j_m+2)(j_m+3)(j_m+4)\cdots
j_{i+1}\overline{1})^{c_{i+1}}
\\
&
(23\cdots (j_m+2)(j_m+3)(j_m+4)\cdots
j_{i}\overline{1})^{\sum_{j=1}^{i}c_j-(c_m-l)/2}
\\
&
(12\cdots (j_m+1)(j_m+3)(j_m+4)\cdots
j_{i}\overline{(j_m+2)})^{(c_m-l)/2-\sum_{j=1}^{i-1}c_j}
\\
&
(12\cdots (j_m+1)(j_m+3)(j_m+4)\cdots
j_{i-1}\overline{(j_m+2)})^{c_{i-1}}
\cdots\cdots
\\
&
(12\cdots (j_m+1)(j_m+3)(j_m+4)\cdots
j_{1}\overline{(j_m+2)})^{c_{1}},
\end{align*}
and $f_{\mathrm{Rev}(\boldsymbol{a}_5)}\circ
\Phi\circ\mathfrak{S}\circ
\Phi^{-1}\circ
e_{\boldsymbol{a}_5}(w_2)$ is
\begin{align*}
&
(456\cdots (j_m+2)\overline{1})^l
(234\cdots (j_m+1)(j_m+2)\overline{1})^{c_m-l}
\\
&
(234\cdots (j_m+1)(j_m+2)(j_m+3)\cdots
j_{m-1}\overline{1})^{c_{m-1}}
\cdots\cdots
\\
&
(234\cdots (j_m+1)(j_m+2)(j_m+3)\cdots
j_{i+1}\overline{1})^{c_{i+1}}
\\
&
(234\cdots (j_m+1)(j_m+2)(j_m+3)\cdots
j_{i}\overline{1})^{\sum_{j=1}^{i}c_j-(c_m-l)/2}
\\
&
(134\cdots (j_m+1)(j_m+2)(j_m+3)\cdots
j_{i}\overline{2})^{(c_m-l)/2-\sum_{j=1}^{i-1}c_j}
\\
&
(134\cdots (j_m+1)(j_m+2)(j_m+3)\cdots
j_{i-1}\overline{2})^{c_{i-1}}
\cdots\cdots
\\
&
(134\cdots (j_m+1)(j_m+2)(j_m+3)\cdots
j_{1}\overline{2})^{c_{1}}.
\end{align*}
{}From this expression, we have $\varepsilon_0(w_2)=2k-c_m,\varphi_0(w_2)=c_m-l$.
Applying
$e_1^{\varepsilon_0(w_2)}$ we get
\begin{align*}
w_3:=&\,
e_1^{2k-c_m}\circ
f_{\mathrm{Rev}(\boldsymbol{a}_5)}\circ
\Phi\circ\mathfrak{S}\circ
\Phi^{-1}\circ
e_{\boldsymbol{a}_5}(w_2)
\\
=&\, (456\cdots (j_m+2)\overline{2})^{l}
(134\cdots (j_m+2)\overline{2})^{c_m-l}
(134\cdots j_{m-1}\overline{2})^{c_{m-1}}
\cdots
(134\cdots j_1\overline{2})^{c_1}.
\end{align*}
Note that the length of the string
$(456\cdots (j_m+2)\overline{2})$ is $j_m$
whereas that of the string
$(134\cdots (j_m+2)\overline{2})$ is $j_m+2$.

In order to convert the action of $e_1^{2k-c_m}$
into that of $e_0^{2k-c_m}$,
we define the word
$\boldsymbol{a}_{6}
=\boldsymbol{a}_{6,1}
\boldsymbol{a}_{6,2}
\boldsymbol{a}_{6,3}
\boldsymbol{a}_{6,4}
\boldsymbol{a}_{6,5}$
as follows.
$\boldsymbol{a}_{6,1}=2^{2k-l}$, $\boldsymbol{a}_{6,2}=
\boldsymbol{a}_{6,2,m+1}
\cdots
\boldsymbol{a}_{6,2,2}$
where subwords are defined by
\begin{align*}
&\boldsymbol{a}_{6,2,m+1}=
3^{2k}4^{2k}\cdots (j_m+1)^{2k},\\
&\boldsymbol{a}_{6,2,m}=
(j_m+2)^{2k-c_m}(j_m+3)^{2k-c_m}\cdots
(j_{m-1}-1)^{2k-c_m},\\
&\boldsymbol{a}_{6,2,i}=
j_i^{2k-\sum_{j=i}^mc_j}
(j_i+1)^{2k-\sum_{j=i}^mc_j}\cdots
(j_{i-1}-1)^{2k-\sum_{j=i}^mc_j}\quad
(m-1\geq i\geq 2),
\end{align*}
$\boldsymbol{a}_{6,3}=
j_1^k(j_1+1)^k\cdots n^k(n-1)^k\cdots (j_1+1)^kj_1^k$,
$\boldsymbol{a}_{6,4}=
\boldsymbol{a}_{6,4,1}\cdots
\boldsymbol{a}_{6,4,m-1}$ where
subwords are defined by
\begin{align*}
&\boldsymbol{a}_{6,4,i}=
(j_i-1)^{k-\sum_{j=1}^ic_j}
(j_i-2)^{k-\sum_{j=1}^ic_j}\cdots
j_{i+1}^{k-\sum_{j=1}^ic_j},\quad
(1\leq i\leq m-2),\\
&\boldsymbol{a}_{6,4,m-1}=
(j_{m-1}-1)^{k-\sum_{j=1}^{m-1}c_j}
(j_{m-1}-2)^{k-\sum_{j=1}^{m-1}c_j}\cdots
(j_{m}+2)^{k-\sum_{j=1}^{m-1}c_j},
\end{align*}
and $\boldsymbol{a}_{6,5}=(j_m+1)^l\cdots 3^l2^l$.
Computation of $e_{\boldsymbol{a}_6}(w_3)$
proceeds as follows:
\begin{align*}
e_{\boldsymbol{a}_{6,1}}(w_3)=&\,
(4567\cdots (j_m+2)\overline{3})^{l}
(1245\cdots (j_m+2)\overline{3})^{c_m-l}
\\&\,
(1245\cdots j_{m-1}\overline{3})^{c_{m-1}}
\cdots
(1245\cdots j_1\overline{3})^{c_1}
\\
\xrightarrow{e_{\boldsymbol{a}_{6,2}}}&\,
(345\cdots (j_m+1)\overline{j_1})^{l}
(123\cdots (j_m+1)\overline{j_1})^{c_m-l}
\\&\,
(123\cdots (j_{m-1}-1)\overline{j_1})^{c_{m-1}}
\cdots
(123\cdots (j_1-1)\overline{j_1})^{c_1}
\\
\xrightarrow{e_{\boldsymbol{a}_{6,3}}}&\,
(345\cdots (j_m+1){j_1})^{l}
(123\cdots (j_m+1){j_1})^{c_m-l}
\\&\,
(123\cdots (j_{m-1}-1){j_1})^{c_{m-1}}
\cdots
(123\cdots (j_1-1){j_1})^{c_1}
\\
\xrightarrow{e_{\boldsymbol{a}_{6,4}}}&\,
(345\cdots (j_m+1)(j_m+2))^{l}
(123\cdots (j_m+1)(j_m+2))^{c_m-l}
\\&\,
(123\cdots (j_{m-1}-1){j_{m-1}})^{c_{m-1}}
\cdots
(123\cdots (j_1-1){j_1})^{c_1}
\\
\xrightarrow{e_{\boldsymbol{a}_{6,5}}}&\,
(234\cdots j_m(j_m+1))^{l}
(123\cdots (j_m+1)(j_m+2))^{c_m-l}
\\&\,
(123\cdots (j_{m-1}-1){j_{m-1}})^{c_{m-1}}
\cdots
(123\cdots (j_1-1){j_1})^{c_1}.
\end{align*}
By $\Phi^{-1}$, this corresponds to the following 
$\pm$-diagram, and
$\mathfrak{S}$ acts on it as follows:
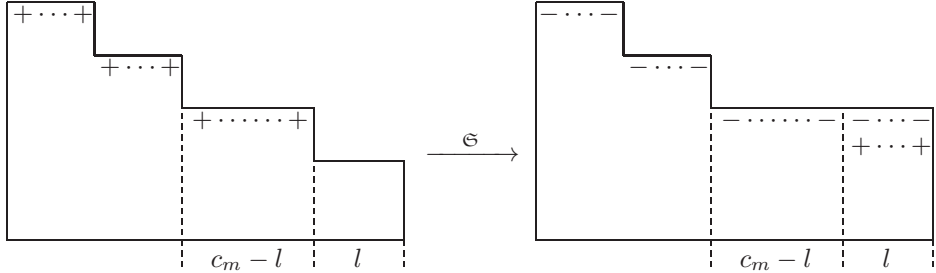
\begin{center}
\unitlength 10pt
\begin{picture}(35,11)
\multiput(0,1)(20,0){2}{
\put(0,0){\line(1,0){15}}
\put(0,0){\line(0,1){9}}
\put(0,9){\line(1,0){3.3}}
\put(3.3,9){\line(0,-1){2}}
\put(3.3,7){\line(1,0){3.3}}
\put(6.6,7){\line(0,-1){2}}
\put(6.6,5){\line(1,0){5}}
\put(15,3){\line(0,-1){3}}
\put(7.7,-1){$c_m-l$}
\put(13.1,-1){$l$}
\multiput(6.6,-1)(0,0.4){15}{\line(0,1){0.2}}
\multiput(11.6,-1)(0,0.4){15}{\line(0,1){0.2}}
\multiput(15,-1)(0,0.4){5}{\line(0,1){0.2}}
}
\put(11.6,6){\line(0,-1){2}}
\put(11.6,4){\line(1,0){3.4}}
\put(35,6){\line(0,-1){2}}
\put(31.6,6){\line(1,0){3.4}}
\put(0.2,9.3){$+\cdots +$}
\put(3.5,7.3){$+\cdots +$}
\put(7.0,5.3){$+\cdots\cdots +$}
\put(15.8,4){$\xrightarrow{\hspace{4mm}\mathfrak{S}\hspace{4mm}}$}
\put(20.1,9.3){$-\cdots -$}
\put(23.5,7.3){$-\cdots -$}
\put(27.0,5.3){$-\cdots\cdots -$}
\put(31.9,5.3){$-\cdots -$}
\put(31.9,4.3){$+\cdots +$}
\end{picture}
\end{center}
Let us assume that $\sum_{j=1}^{i-1}2c_j
<l\leq\sum_{j=1}^i2c_j$.
Then the right $\pm$-diagram corresponds to
the expression (\ref{eq:w_3}) with $\sum_{j=1}^ic_j-c_m/2$
and $c_m/2-\sum_{j=1}^{i-1}c_j$
in (\ref{eq:w_3})
being replaced with $\sum_{j=1}^ic_j-l/2$
and $l/2-\sum_{j=1}^{i-1}c_j$.
Application of $f_{\mathrm{Rev}(\boldsymbol{a}_{6,5})}$
is similar to that of $f_{\mathrm{Rev}(\boldsymbol{a}_{3})}$
on (\ref{eq:w_3}) and we obtain
$f_{\mathrm{Rev}(\boldsymbol{a}_{6,5})}\circ
\Phi\circ\mathfrak{S}\circ\Phi^{-1}\circ
e_{\boldsymbol{a}_{6}}(w_{3})$
as
\begin{align}\label{eq:special3}
&
(234\cdots (j_m+2)\overline{1})^{c_m}
(234\cdots j_{m-1}\overline{1})^{c_{m-1}}
\cdots
(234\cdots j_{i-1}\overline{1})^{c_i}
\\
&
(234\cdots j_i\overline{1})^{\sum_{j=1}^ic_j-l/2}
(134\cdots j_i\overline{2})^{l/2-\sum_{j=1}^{i-1}c_j}
\cdots
(134\cdots j_2\overline{2})^{c_2}
(134\cdots j_1\overline{2})^{c_1}.
\nonumber
\end{align}
The remaining computation of
$f_{\mathrm{Rev}(\boldsymbol{a}_6)}$ is almost the same as
the computation of $f_{\mathrm{Rev}(\boldsymbol{a}_4)}$
given in the final part of the proof of
Lemma \ref{lem:special1}.
The only difference in
$f_{\boldsymbol{a}_{6,1}}=f_2^{2k-l}$
is caused by the fact that letters $1$ and $\overline{2}$
appear $l$ times in $f_{\mathrm{Rev}(\boldsymbol{a}_{6,5})}\circ
\Phi\circ\mathfrak{S}\circ\Phi^{-1}\circ
e_{\boldsymbol{a}_{6}}(w_{3})$.
\end{proof}

As for $e_{\boldsymbol{a}_2}$, the beginning two steps
$e_1^{2k}e_2^{2k}
e_0^{2k-c_m}f_{\boldsymbol{a}_1}
(1^l\otimes\mu)$
gives
\begin{align}
1^l\otimes\,
&(145\cdots (j_m+2)\overline{3}\,\overline{2})^{c_m}
(145\cdots j_{m-1}\overline{3}\,\overline{2})^{c_{m-1}}
\cdots
(145\cdots j_2\overline{3}\,\overline{2})^{c_2}
(145\cdots j_1\overline{3}\,\overline{2})^{c_1}.
\nonumber
\end{align}
By comparing this with
$e_1^{2k}e_2^{2k+\delta}e_0^{2k-c_m+\delta}
f_{\boldsymbol{a}_1}(\mu\otimes 1^l)$,
we see that the rest of the computation of
$e_{\boldsymbol{a}_2}$ is almost the same as that
given in Lemma \ref{lem:special2}.
This completes the proof for Case 1.

\subsubsection{Proof for Case 2.}
Note that in this case we have
$(l-c_m)_+=l-c_m$.
Action of $e_0^{2k-c_m}$ is obtained by
formally setting $l=c_m$ in Case 1.
Therefore we have
\begin{align}
e_0^{2k-c_m+(l-c_m)_+}
f_{\boldsymbol{a}_1}
(1^l\otimes\mu)
=\overline{2}
\mathstrut^{l-c_m}3^{c_m}
\otimes
(e_0^{2k-c_m}(w_2)|_{l=c_m}).
\nonumber
\end{align}
where $e_0^{2k-c_m}(w_2)$ is given in (\ref{eq:special3}).
When we further apply
$e_{\boldsymbol{a}_2}$ on this formula,
we realize that there are
extra exponents originating from
$\overline{2}
\mathstrut^{l-c_m}$
in the first tensor component of the
right hand side.
These extra contributions coincide with the exponents $\delta$
in $\boldsymbol{a}_2$. We have completed the proof of
Lemma \ref{lem_special2}.

\subsection{Proof: Part 3}
Now we can prove Proposition \ref{prop_special}. We prove the first relation by descending induction
on $j_m$. If $j_m=r$ (this is the maximal possible value), we have $\mu=(k^r)$. In this case we see
$((k^r)\ot 1^l)=1^l\ot(k^r)$ by weight consideration. The induction proceeds by using Lemmas 
\ref{lem_special1} and \ref{lem_special2}. 

As for the energy function, we have to
look carefully the action of $e_0$
in Lemmas \ref{lem_special1} and \ref{lem_special2}.
If $e_0$ acts on the second component of
the tensor product, we write $R$,
and $L$ on the first component.
We summarize actions of $e_0$
to get $\bar{\mu}\otimes 1^l$
and $1^l\otimes\bar{\mu}$
in two lemmas as follows
(proceeds from left to right):
\begin{align}
\mu\otimes 1^l:\,&
\underbrace{R\cdots\cdots\cdots R}_{(l-c_m)_+}
\underbrace{L\cdots L}_{2k-c_m}
\nonumber\\
1^l\otimes\mu :\,&
\underbrace{R\cdots R}_{2k-c_m}
\underbrace{L\cdots\cdots\cdots L}_{(l-c_m)_+}
\nonumber
\end{align}
The diagram is drawn in the case of $(l-c_m)_+>2k-c_m$.
Including the other inequality case,
we see that we have exactly the same
number of $(LL)$ and $(RR)$ cases (see \eqref{eq:e-func}).
Therefore we have $H(\mu\otimes 1^l)=H(\bar{\mu}\otimes 1^l)$.
Using the same induction as above we obtain
$H(\mu\otimes 1^l)=H((k^r)\otimes 1^l)=0$.
This completes the proof of
Proposition \ref{prop_special}.

\section{Reduction to the special case}

\subsection{Odd $r$ case}

\subsubsection{Calculation in  $B^{r,k}\otimes B^{1,l}$}
Let $\mu\ot x\in B^{r,k}\ot B^{1,l}$ be $I_0$-highest. Recall that we defined $\mu_i$ by 
$\mu=\sum_i\mu_i\La_i$. (Readers are warned that it is not the multiplicity of $i$ in the corresponding
partition $\mu$ but its conjugate $\mu'$.) Note that $\mu_i=0$ unless $1\le i\le r$ and $i$ is odd. We also 
know that the coordinates other than $x_1,x_2,\ldots,x_{r+1},\bar{x}_r,
\ldots,\bar{x}_3,\bar{x}_1$ are all $0$ by Proposition \ref{prop:ht cond 1}.

Let us define a word 
$\boldsymbol{b}_1=\boldsymbol{b}_{1,1}
\boldsymbol{b}_{1,2}$ by 
\begin{align}
\boldsymbol{b}_{1,1}=
&\,
2^{s_2}3^{s_3}\cdots r^{s_r}
(r+1)^{\alpha}(r+2)^{\alpha}\cdots n^{\alpha}(n-1)^{\alpha}
\cdots (r+1)^{\alpha}r^{\alpha}
\label{b11}\\
&
(r-1)^{\bar{s}_{r-1}}(r-2)^{\bar{s}_{r-2}}\cdots
2^{\bar{s}_{2}},
\nonumber\\
\boldsymbol{b}_{1,2}=
&\,
1^{\bar{\bar{s}}_1}
2^{\bar{\bar{s}}_2}3^{\bar{\bar{s}}_3}\cdots
(r-1)^{\bar{\bar{s}}_{r-1}}
r^k(r+1)^k\cdots n^k(n-1)^k\cdots (r+1)^kr^k
\nonumber\\
&
(r-1)^{\bar{\bar{\bar{s}}}_{r-1}}
(r-2)^{\bar{\bar{\bar{s}}}_{r-2}}\cdots
2^{\bar{\bar{\bar{s}}}_{2}}
1^{\bar{\bar{\bar{s}}}_{1}}
\nonumber
\end{align}
where the exponents are defined as follows.
For $i=1,2,\cdots,(r+1)/2$,
\begin{align}
s_{2i-1}=&\,
2k-\sum_{j=1:\mathrm{odd}}^{2i-1}\mu_j
+x_1+x_{2i}+\sum_{j=3:\mathrm{odd}}^{2i-1}\bar{x}_{j},
\qquad
s_{2i}=s_{2i-1}-x_{2i}.
\nonumber
\end{align}
Define $\alpha=k+x_1+\sum_{j=3:\mathrm{odd}}^{r}\bar{x}_{j}$.
Then $s_r=\alpha+x_{r+1}$.
For $i=(r-1)/2,\cdots,2,1$,
\begin{align}
\bar{s}_{2i}=\bar{s}_{2i-1}=
\alpha-\sum_{j=2i+1:\mathrm{odd}}^r\mu_j
+\sum_{j=2i+1}^{r+1}x_j.
\nonumber
\end{align}
Set $\bar{\bar{s}}_1=2k+l-\mu_1$
and define other $\bar{\bar{s}}_i$ by
$\bar{\bar{s}}_{2i-1}=\bar{\bar{s}}_{2i}=
2k-\sum_{j=1:\mathrm{odd}}^{2i-1}\mu_j$
for $i=1,2,\cdots,(r-1)/2$.
Note that $\bar{\bar{s}}_{r-1}=k+\mu_r$.
Set ${\bar{\bar{\bar{s}}}_{1}}=0$,
${\bar{\bar{\bar{s}}}_{2}}=x_2+\bar{x}_1$
and define other
${\bar{\bar{\bar{s}}}_{i}}$ by
${\bar{\bar{\bar{s}}}_{2i}}=
{\bar{\bar{\bar{s}}}_{2i-1}}=
k-\sum_{j=2i+1:\mathrm{odd}}^r\mu_j$
for $i=(r-1)/2,\cdots,3,2$.

The goal of this subsection is to prove the following proposition.

\begin{prop}\label{reduction_odd}
We have
$$e_{\boldsymbol{b}_1}e_0^{2k-\mu_1+x_1+x_2}(\mu\otimes x)
=\bar{\mu}\otimes \bar{3}^{x_2+\bar{x}_1}1^{l-x_2-\bar{x}_1}$$
where $\bar{\mu}=\sum_{i=3:\mathrm{odd}}^r\mu_i\La_i+\mu_1\La_3$.
\end{prop}

During the proof, we assume $\mu_1$ even.
The proof for odd $\mu_1$ is similar.

\begin{lemma}
If $\sum_{j=i+1}^{r}2\mu_j<\mu_1\leq
\sum_{j=i}^{r}2\mu_j$, then
\begin{align*}
e_0^{2k-\mu_1}\mu =&\,
(3\bar{2}\bar{1})^{\mu_1+\mu_3}
(345\bar{2}\bar{1})^{\mu_5}\cdots
(345\cdots (i-2)\bar{2}\bar{1})^{\mu_{i-2}}
(345\cdots i\bar{2}\bar{1})^{\sum_{j=i}^{r}\mu_j-\mu_1/2}\\
&\,
(145\cdots i\bar{3}\bar{2})^{\mu_1/2-\sum_{j=i+1}^{r}\mu_j}
(145\cdots (i+2)\bar{3}\bar{2})^{\mu_{i+2}}
\cdots (145\cdots r\bar{3}\bar{2})^{\mu_r}.
\end{align*}
\end{lemma}

\begin{proof}
Since $\Phi^{-1}(\mu)$ is the $\pm$-diagram of outer shape $\mu$ such that all the columns have $+$ 
as symbol, we see $\Phi\circ\mathfrak{S}\circ\Phi^{-1}(\mu)
=\bar{1}^{\mu_1}(23\bar{1})^{\mu_3}\cdots
(23\cdots r\bar{1})^{\mu_r}$.
Thus one has $\varepsilon_0(\mu)=2k-\mu_1,\varphi_0(\mu)=0$.
We have
$e_1^{2k-\mu_1}\circ\Phi\circ\mathfrak{S}\circ\Phi^{-1}(\mu)
=\bar{2}^{\mu_1}
(13\bar{2})^{\mu_3}\cdots
(134\cdots r\bar{2})^{\mu_r}$.
To convert the result into that for
$e_0^{2k-\mu_1}$ we define a word 
$\boldsymbol{b}_{1.1}'$ as follows:
\begin{align}
\boldsymbol{b}_{1,1}'=
&\,
2^{s_2'}3^{s_3'}\cdots r^{s_r'}
(r+1)^{\alpha'}(r+2)^{\alpha'}\cdots n^{\alpha'}(n-1)^{\alpha'}
\cdots (r+1)^{\alpha'}r^{\alpha'}
\nonumber\\
&
(r-1)^{\bar{s}_{r-1}'}(r-2)^{\bar{s}_{r-2}'}\cdots
2^{\bar{s}_{2}'}
\nonumber
\end{align}
where $s_i'=s_i|_{x_j=0}$,
$\alpha'=\alpha|_{x_j=0}$, $\bar{s}_i'=\bar{s}_i|_{x_j=0}$.
Then we have
\begin{align*}
e_{\boldsymbol{b}_{1,1}'}\circ
e_1^{2k-\mu_1}\circ\Phi\circ\mathfrak{S}\circ\Phi^{-1}(\mu)
=
2^{\mu_1}(123)^{\mu_3}(12345)^{\mu_5}\cdots
(12\cdots r)^{\mu_r}.
\end{align*}
Applying $\Phi\circ\mathfrak{S}\circ\Phi^{-1}$ further,
we obtain
\begin{align}\label{tochu6_1}
&(23\bar{1})^{\mu_1+\mu_3}
(2345\bar{1})^{\mu_5}\cdots
(23\cdots (i-2)\bar{1})^{\mu_{i-2}}
(23\cdots i\bar{1})^{\sum_{j=i}^{r}\mu_j-\mu_1/2}
\\
&(1245\cdots i\bar{3})^{\mu_1/2-\sum_{j=i+1}^{r}\mu_j}
(1245\cdots (i+2)\bar{3})^{\mu_{i+2}}\cdots
(1245\cdots r\bar{3})^{\mu_r}.
\nonumber
\end{align}
Finally, applying $f_{\mathrm{Rev}(\boldsymbol{b}_{1,1}')}$ we obtain the desired relation.
\end{proof}

\begin{lemma}
$e_0^{2k-\mu_1+x_1+x_2}(\mu\otimes x)$ is equal to
\begin{align}\label{tochu6_2}
e_0^{2k-\mu_1}\mu\otimes
\bar{1}^{x_2+\bar{x}_1}
\bar{2}^{x_1}
\bar{3}^{\bar{x}_3}\cdots
\overline{(r-2)}^{\bar{x}_{r-2}}
\bar{r}^{\bar{x}_r}
(r+1)^{x_{r+1}}\cdots
4^{x_4}
3^{x_3}
\end{align}
\end{lemma}

\begin{proof}
Note that the second component of the RHS is
$e_0^{x_1+x_2}x$.
Since $\varphi_0(\mu)=0$,
we obtain the expression.
\end{proof}

\begin{lemma}
$e_{\boldsymbol{b}_{1,1}}
e_0^{2k-\mu_1+x_1+x_2}(\mu\otimes x)=$
Eq.$(\ref{tochu6_1})\otimes
\bar{1}^{x_2+\bar{x}_1}2^{l-x_2-\bar{x}_1}$.
\end{lemma}

\begin{proof}
Let us consider the operation of $e_2^{s_2}=e_2^{2k-\mu_1+x_1}$ in
$e_{\boldsymbol{b}_{1,1}}$.
Since $\varphi_2(e_0^{2k-\mu_1}(\mu))=0$,
$e_2$ acts on the second component at most $\ve_2(\text{2nd comp})$ times 
and the rest goes to the first.
The 2-signature of the second component of (\ref{tochu6_2})
is $-^{x_1}+^{\bar{x}_3}-^{x_3}$.
{}From the highest condition for $\mu\otimes x$
we have $\bar{x}_3\geq x_3$, thus $e_2$ acts
on $\bar{2}^{x_1}$ only.
We can continue similarly and obtain the desired result.
\end{proof}

Finally, we consider the action of $e_{\boldsymbol{b}_{1,2}}$.
The 1-signature of Eq.$(\ref{tochu6_1})\otimes
\bar{1}^{x_2+\bar{x}_1}2^{l-x_2-\bar{x}_1}$
is $-^{2k-\mu_1}(+-)^{\mu_1/2}-^l$.
By applying $e_1^{2k+l-\mu_1}$, we get
\begin{align}
&(13\bar{2})^{\mu_1+\mu_3}
(1345\bar{2})^{\mu_5}\cdots
(1345\cdots (i-2)\bar{2})^{\mu_{i-2}}
(1345\cdots i\bar{2})^{\sum_{j=i}^{r}\mu_j-\mu_1/2}
\nonumber\\
&(1245\cdots i\bar{3})^{\mu_1/2-\sum_{j=i+1}^{r}\mu_j}
(1245\cdots (i+2)\bar{3})^{\mu_{i+2}}\cdots
(1245\cdots r\bar{3})^{\mu_r}
\otimes \bar{2}^{x_2+\bar{x}_1}1^{l-x_2-\bar{x}_1}.
\nonumber
\end{align}
The 2-signature of the above element is
$-^{2k-\mu_1}+^{\mu_1}-^{x_2+\bar{x}_1}$.
{}From the highest condition for $\mu\otimes x$
we have $\mu_1\geq x_2+\bar{x}_1$, thus
$e_2$ does not act on
$\bar{2}^{x_2+\bar{x}_1}1^{l-x_2-\bar{x}_1}$.
Therefore $e_2^{2k-\mu_1}$ acts on the first component and obtain
\begin{align}
&(12\bar{3})^{\mu_1+\mu_3}
(1245\bar{3})^{\mu_5}\cdots
(1245\cdots (i-2)\bar{3})^{\mu_{i-2}}
(1245\cdots i\bar{3})^{\sum_{j=i}^{r}\mu_j-\mu_1/2}
\nonumber\\
&(1245\cdots i\bar{3})^{\mu_1/2-\sum_{j=i+1}^{r}\mu_j}
(1245\cdots (i+2)\bar{3})^{\mu_{i+2}}\cdots
(1245\cdots r\bar{3})^{\mu_r}
\otimes \bar{2}^{x_2+\bar{x}_1}1^{l-x_2-\bar{x}_1}.
\nonumber
\end{align}
We can continue the computation and
arrive at Proposition \ref{reduction_odd}.

\subsubsection{Calculation in $B^{1,l}\otimes B^{r,k}$}
In this subsection, let $P$ and $P'$ be the $\pm$-diagrams.
As before, corresponding to $P$ and $P'$,
we use the parametrization $p_i^\ast$ and
$p_i'{}^\ast$ ($\ast =\cdot,+,-,\mp$) respectively.
Note that by definition $p_1^\mp=p_1'{}^\mp=0$.
Define a word $\boldsymbol{b}_2'=\boldsymbol{b}_{2,1}'
\boldsymbol{b}_{2,2}'$ by
\begin{align}
\boldsymbol{b}_{2,1}'=
&\,
2^{t_2}3^{t_3}\cdots r^{t_r}
(r+1)^{\beta}(r+2)^{\beta}\cdots n^{\beta}(n-1)^{\beta}
\cdots (r+1)^{\beta}r^{\beta}
\nonumber\\
&
(r-1)^{\bar{t}_{r-1}}(r-2)^{\bar{t}_{r-2}}\cdots
2^{\bar{t}_{2}},
\nonumber\\
\boldsymbol{b}_{2,2}'=
&\,
1^{\bar{\bar{t}}_1}
2^{\bar{\bar{t}}_2}3^{\bar{\bar{t}}_3}\cdots
(r-1)^{\bar{\bar{t}}_{r-1}}
r^{\beta'}(r+1)^{\beta'}\cdots
n^{\beta'}(n-1)^{\beta'}\cdots (r+1)^{\beta'}r^{\beta'}
\nonumber\\
&
(r-1)^{\bar{\bar{\bar{t}}}_{r-1}}
(r-2)^{\bar{\bar{\bar{t}}}_{r-2}}\cdots
2^{\bar{\bar{\bar{t}}}_{2}}.
%1^{\bar{\bar{\bar{t}}}_{1}}.
\nonumber
\end{align}
Here the exponents for $\boldsymbol{b}_{2,1}'$
are
\begin{align*}
&t_2=k+\sum_{i=1:\mathrm{odd}}^r
(p_i^+-p_i^-)-p_3^\mp-p_1^+,\\
&t_{i+1}-t_i=
\begin{cases}
-p_{i+1}^+ & (\mbox{if } i \mbox{ is even})\\
-p_{i+2}^\mp & (\mbox{if } i \mbox{ is odd}),
\end{cases}\\
&\beta -t_r=-p_r^\cdot ,\qquad
\bar{t}_{r-1}-\beta =-p_r^+-p_{r-2}^\cdot ,\\
&\bar{t}_{i}-\bar{t}_{i+1}=
\begin{cases}
-p_{i+1}^+-p_{i-1}^\cdot & (\mbox{if } i \mbox{ is even})\\
0 & (\mbox{if } i \mbox{ is odd}).
\end{cases}
\end{align*}
The exponents for $\boldsymbol{b}_{2,2}'$
are
\begin{align*}
&\bar{\bar{t}}_1=2k-p^\mp_3-p_1^--p_1^+\\
&\bar{\bar{t}}_{i+1}-\bar{\bar{t}}_i=
\begin{cases}
-p^\mp_{i+3}-p^-_{i+1}-p^+_{i+1}-p^\cdot_{i-1}
& (\mbox{if } i \mbox{ is even})\\
0 & (\mbox{if } i \mbox{ is odd}),
\end{cases}\\
& \beta'-\bar{\bar{t}}_{r-1}=
\bar{\bar{\bar{t}}}_{r-1}-\beta'=
-p^\mp_{r+2}-p^-_{r}-p^+_{r}-p^\cdot_{r-2},\\
& \bar{\bar{\bar{t}}}_{i}-\bar{\bar{\bar{t}}}_{i+1}=
\begin{cases}
-p^\mp_{i+3}-p^-_{i+1}-p^+_{i+1}-p^\cdot_{i-1}
& (\mbox{if } i \mbox{ is even and }i\neq 2)\\
-p^\mp_{5}-p^-_{3}-p^+_{3}-p^+_1-p^\cdot_{1}
& (\mbox{if } i=2)\\
0 & (\mbox{if } i \mbox{ is odd}).
\end{cases}
\end{align*}

{}From $\boldsymbol{b}_2'$,
we define another word
$\boldsymbol{b}_2=\boldsymbol{b}_{2,1}
\boldsymbol{b}_{2,2}$ as follows:
\begin{align}
\boldsymbol{b}_{2,1}=
&\,
2^{t_2+l}3^{t_3+l}\cdots r^{t_r+l}
(r+1)^{\beta +l}(r+2)^{\beta +l}\cdots n^{\beta +l}(n-1)^{\beta +l}
\cdots (r+1)^{\beta +l}r^{\beta +l}
\nonumber\\
&
(r-1)^{\bar{t}_{r-1}+l}(r-2)^{\bar{t}_{r-2}+l}\cdots
2^{\bar{t}_{2}+l},
\nonumber\\
\boldsymbol{b}_{2,2}=
&\,
1^{\bar{\bar{t}}_1+l}
2^{\bar{\bar{t}}_2}3^{\bar{\bar{t}}_3}\cdots
(r-1)^{\bar{\bar{t}}_{r-1}}
r^{\beta'}(r+1)^{\beta'}\cdots
n^{\beta'}(n-1)^{\beta'}\cdots (r+1)^{\beta'}r^{\beta'}
\nonumber\\
&
(r-1)^{\bar{\bar{\bar{t}}}_{r-1}}
(r-2)^{\bar{\bar{\bar{t}}}_{r-2}}\cdots
2^{\bar{\bar{\bar{t}}}_{2}}.
%1^{\bar{\bar{\bar{t}}}_{1}}.
\nonumber
\end{align}
Then the goal of this subsection is to prove the following proposition:

\begin{prop}\label{reduction_odd2}
We have 
\begin{align}
e_{\boldsymbol{b}_2}e_0^{\varepsilon_0(P)+l}(1^l\otimes P)
=1^l\otimes P'
\nonumber
\end{align}
where
\begin{align*}
{\varepsilon_0(P)}=&\,
\sum_j(p^\cdot_j +2p^+_j+p^\mp_j)-p_1^+,
\end{align*}
and $P'$ is related with $P$ as
\begin{align*}
&p_1'{}^\cdot =p_3^\mp +p_1^-,&
&p_3'{}^+=p_5^\mp +p_3^- +p_3^+ +p_1^\cdot +p_1^+,\\
&p_r'{}^+=p_r^\cdot +p_r^-+p_r^+ +p_{r-2}^\cdot ,&
&p_i'{}^+=p_{i+2}^\mp +p_i^-+p_i^++p_{i-2}^{\cdot},
\end{align*}
where $i$ is an odd integer such that $3<i<r$ 
and all other $p_i'{}^\ast =0$.
\end{prop}

Since $e_0^{\varepsilon_0(P)+l}(1^l\otimes P)=
\bar{2}^l\otimes e_0^{\varepsilon_0(P)}(P)$,
this proposition is the consequence of the
following lemma:

\begin{lemma}
With the notations of Proposition \ref{reduction_odd2},
we have
$e_{\boldsymbol{b}_2'}e_0^{\varepsilon_0(P)}(P)=P'$.
\end{lemma}

The rest of this subsection is devoted to the proof
of this lemma. To begin with, we observe the following:

\begin{lemma}
${\varepsilon_0(P)}=
\sum_j(p^\cdot_j +2p^+_j+p^\mp_j)-p_1^+.$
\end{lemma}

\begin{proof}
We use Proposition \ref{prop:e1 action}.
Schematically, the pair of $\pm$-diagrams
%$(\mathfrak{S}(P),p^s)$
corresponding to $\mathfrak{S}(P)$
looks as follows:
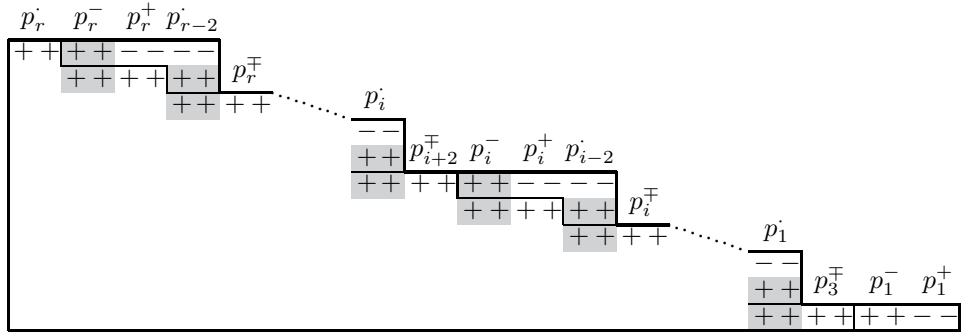
\begin{align}
\unitlength 10pt
\begin{picture}(36,12.5)
\color[cmyk]{0,0,0,0.2}
\put(2,9){\rule{20pt}{20pt}}
\put(6,8){\rule{20pt}{20pt}}
\put(13,5){\rule{20pt}{20pt}}
\put(17,4){\rule{20pt}{20pt}}
\put(21,3){\rule{20pt}{20pt}}
\put(28,0){\rule{20pt}{20pt}}
\color{black}
\thicklines
%%%outer
\put(0,0){\line(1,0){36}}
\put(0,0){\line(0,1){11}}
\put(0,11){\line(1,0){8}}
\put(8,11){\line(0,-1){2}}
\put(8,9){\line(1,0){2}}
\multiput(10,9)(0.3,-0.1){10}{\circle*{0.1}}
\put(13,8){\line(1,0){2}}
\put(15,8){\line(0,-1){2}}
\put(15,6){\line(1,0){8}}
\put(23,6){\line(0,-1){2}}
\put(23,4){\line(1,0){2}}
\multiput(25,4)(0.3,-0.1){10}{\circle*{0.1}}
\put(28,3){\line(1,0){2}}
\put(30,3){\line(0,-1){2}}
\put(30,1){\line(1,0){6}}
\put(36,1){\line(0,-1){1}}
%%%inner
\thinlines
\put(2,11){\line(0,-1){1}}
\put(2,10){\line(1,0){4}}
\put(6,10){\line(0,-1){1}}
\put(6,9){\line(1,0){2}}
%
%\put(13,7){\line(0,-1){1}}
\put(13,6){\line(1,0){2}}
\put(17,6){\line(0,-1){1}}
\put(17,5){\line(1,0){4}}
\put(21,5){\line(0,-1){1}}
\put(21,4){\line(1,0){2}}
\put(28,1){\line(1,0){2}}
\put(32,1){\line(0,-1){1}}
%%%
\put(0.2,10.3){$+$}
\put(1.2,10.3){$+$}
\put(2.2,10.3){$+$}
\put(3.2,10.3){$+$}
\put(4.2,10.3){$-$}
\put(5.2,10.3){$-$}
\put(6.2,10.3){$-$}
\put(7.1,10.3){$-$}
\put(2.2,9.3){$+$}
\put(3.2,9.3){$+$}
\put(4.2,9.3){$+$}
\put(5.2,9.3){$+$}
\put(6.2,9.3){$+$}
\put(7.1,9.3){$+$}
\put(6.2,8.3){$+$}
\put(7.1,8.3){$+$}
\put(8.2,8.3){$+$}
\put(9.2,8.3){$+$}
\put(13.2,7.3){$-$}
\put(14.1,7.3){$-$}
\put(13.2,6.3){$+$}
\put(14.1,6.3){$+$}
\put(13.2,5.3){$+$}
\put(14.1,5.3){$+$}
\put(15.2,5.3){$+$}
\put(16.2,5.3){$+$}
\put(17.2,5.3){$+$}
\put(18.2,5.3){$+$}
\put(19.2,5.3){$-$}
\put(20.2,5.3){$-$}
\put(21.2,5.3){$-$}
\put(22.2,5.3){$-$}
\put(17.2,4.3){$+$}
\put(18.2,4.3){$+$}
\put(19.2,4.3){$+$}
\put(20.2,4.3){$+$}
\put(21.2,4.3){$+$}
\put(22.2,4.3){$+$}
\put(21.2,3.3){$+$}
\put(22.2,3.3){$+$}
\put(23.2,3.3){$+$}
\put(24.2,3.3){$+$}
\put(28.2,2.3){$-$}
\put(29.2,2.3){$-$}
\put(28.2,1.3){$+$}
\put(29.2,1.3){$+$}
\put(28.2,0.3){$+$}
\put(29.2,0.3){$+$}
\put(30.2,0.3){$+$}
\put(31.2,0.3){$+$}
\put(32.2,0.3){$+$}
\put(33.2,0.3){$+$}
\put(34.2,0.3){$-$}
\put(35.2,0.3){$-$}
\put(0.5,11.6){$p_r^\cdot$}
\put(2.5,11.6){$p_r^-$}
\put(4.5,11.6){$p_r^+$}
\put(6.0,11.6){$p_{r-2}^\cdot$}
\put(8.5,9.6){$p_{r}^\mp$}
\put(13.5,8.6){$p_i^\cdot$}
\put(15.2,6.6){$p_{i+2}^\mp$}
\put(17.5,6.6){$p_i^-$}
\put(19.5,6.6){$p_i^+$}
\put(21.1,6.6){$p_{i-2}^\cdot$}
\put(23.5,4.6){$p_i^\mp$}
\put(28.6,3.6){$p_1^\cdot$}
\put(30.5,1.6){$p_3^\mp$}
\put(32.6,1.6){$p_1^-$}
\put(34.6,1.6){$p_1^+$}
\end{picture}
\nonumber
\end{align}
Here the thick lines represent outer shape of
$\mathfrak{S}(P)$ and the thin lines represent
the inner $\pm$-diagram.
(Since we are to consider the $e_1$ action, we need such a pair of $\pm$-diagrams.)
The numbers $p_i^\ast$ represent the numbers of
columns which have the same pattern of $+$ and $-$
indicated below $p_i^\ast$.
According to Proposition \ref{prop:e1 action},
we make pairs of two $+$ symbols which we indicate
by gray squares in the diagram.
Then we see that we can apply $e_1$ up to
$\sum_j(p^\cdot_j +2p^+_j+p^\mp_j)-p_1^+$ times,
which gives the value for $\varepsilon_0(P)$.
The pair of $\pm$-diagrams corresponding to
$e_1^{\varepsilon_0(P)}\circ\mathfrak{S}(P)$ looks as follows:
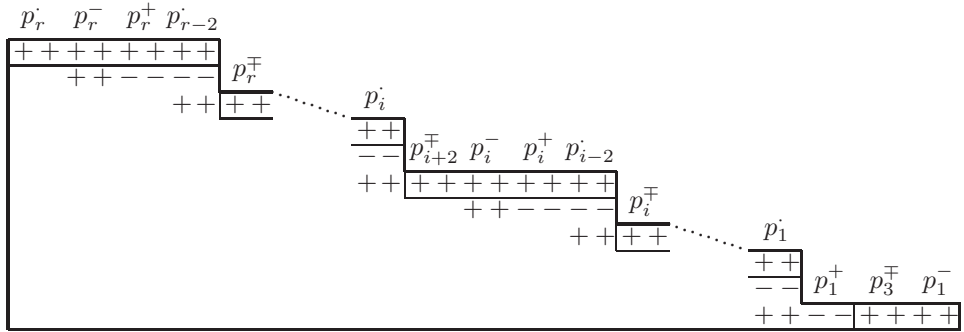
\begin{align}
\unitlength 10pt
\begin{picture}(36,12.5)
\thicklines
%%%outer
\put(0,0){\line(1,0){36}}
\put(0,0){\line(0,1){11}}
\put(0,11){\line(1,0){8}}
\put(8,11){\line(0,-1){2}}
\put(8,9){\line(1,0){2}}
\multiput(10,9)(0.3,-0.1){10}{\circle*{0.1}}
\put(13,8){\line(1,0){2}}
\put(15,8){\line(0,-1){2}}
\put(15,6){\line(1,0){8}}
\put(23,6){\line(0,-1){2}}
\put(23,4){\line(1,0){2}}
\multiput(25,4)(0.3,-0.1){10}{\circle*{0.1}}
\put(28,3){\line(1,0){2}}
\put(30,3){\line(0,-1){2}}
\put(30,1){\line(1,0){6}}
\put(36,1){\line(0,-1){1}}
%%%inner
\thinlines
\put(0,10){\line(1,0){2}}
\put(2,10){\line(1,0){4}}
%\put(6,10){\line(0,-1){1}}
\put(6,10){\line(1,0){2}}
\put(8,9){\line(0,-1){1}}
\put(8,8){\line(1,0){2}}
%
%\put(13,7){\line(0,-1){1}}
\put(13,7){\line(1,0){2}}
\put(15,6){\line(0,-1){1}}
\put(15,5){\line(1,0){6}}
%\put(21,5){\line(0,-1){1}}
\put(21,5){\line(1,0){2}}
\put(23,4){\line(0,-1){1}}
\put(23,3){\line(1,0){2}}
\put(28,2){\line(1,0){2}}
\put(32,1){\line(0,-1){1}}
%%%
\put(0.2,10.3){$+$}
\put(1.2,10.3){$+$}
\put(2.2,10.3){$+$}
\put(3.2,10.3){$+$}
\put(4.2,10.3){$+$}
\put(5.2,10.3){$+$}
\put(6.2,10.3){$+$}
\put(7.1,10.3){$+$}
\put(2.2,9.3){$+$}
\put(3.2,9.3){$+$}
\put(4.2,9.3){$-$}
\put(5.2,9.3){$-$}
\put(6.2,9.3){$-$}
\put(7.1,9.3){$-$}
\put(6.2,8.3){$+$}
\put(7.1,8.3){$+$}
\put(8.2,8.3){$+$}
\put(9.2,8.3){$+$}
\put(13.2,7.3){$+$}
\put(14.1,7.3){$+$}
\put(13.2,6.3){$-$}
\put(14.1,6.3){$-$}
\put(13.2,5.3){$+$}
\put(14.1,5.3){$+$}
\put(15.2,5.3){$+$}
\put(16.2,5.3){$+$}
\put(17.2,5.3){$+$}
\put(18.2,5.3){$+$}
\put(19.2,5.3){$+$}
\put(20.2,5.3){$+$}
\put(21.2,5.3){$+$}
\put(22.2,5.3){$+$}
\put(17.2,4.3){$+$}
\put(18.2,4.3){$+$}
\put(19.2,4.3){$-$}
\put(20.2,4.3){$-$}
\put(21.2,4.3){$-$}
\put(22.2,4.3){$-$}
\put(21.2,3.3){$+$}
\put(22.2,3.3){$+$}
\put(23.2,3.3){$+$}
\put(24.2,3.3){$+$}
\put(28.2,2.3){$+$}
\put(29.2,2.3){$+$}
\put(28.2,1.3){$-$}
\put(29.2,1.3){$-$}
\put(28.2,0.3){$+$}
\put(29.2,0.3){$+$}
\put(30.2,0.3){$-$}
\put(31.2,0.3){$-$}
\put(32.2,0.3){$+$}
\put(33.2,0.3){$+$}
\put(34.2,0.3){$+$}
\put(35.2,0.3){$+$}
\put(0.5,11.6){$p_r^\cdot$}
\put(2.5,11.6){$p_r^-$}
\put(4.5,11.6){$p_r^+$}
\put(6.0,11.6){$p_{r-2}^\cdot$}
\put(8.5,9.6){$p_{r}^\mp$}
\put(13.5,8.6){$p_i^\cdot$}
\put(15.2,6.6){$p_{i+2}^\mp$}
\put(17.5,6.6){$p_i^-$}
\put(19.5,6.6){$p_i^+$}
\put(21.1,6.6){$p_{i-2}^\cdot$}
\put(23.5,4.6){$p_i^\mp$}
\put(28.6,3.6){$p_1^\cdot$}
\put(30.5,1.6){$p_1^+$}
\put(32.6,1.6){$p_3^\mp$}
\put(34.6,1.6){$p_1^-$}
\end{picture}
\nonumber
\end{align}
Note that the numbers of columns of height 1 have changed
from $p_3^\mp$, $p_1^-$, $p_1^+$ to
$p_1^+$, $p_3^\mp$, $p_1^-$.
\end{proof}

In order to compute $e_0^{\varepsilon_0(P)}(P)$,
we usually make $e_1^{\varepsilon_0(P)}\circ\mathfrak{S}(P)$ into
$\{2,3,\cdots,n\}$-highest by applying suitable $e_{\boldsymbol{a}}$, apply $\mathfrak{S}$
and then apply $f_{\mathrm{Rev}(\boldsymbol{a})}$ (see \eqref{eq:sigma}).
However, since $\mathfrak{S}$ commutes with
the action of $e_i$ $(i=2,3,\cdots,n)$, we can apply
$\mathfrak{S}$ on the pair of $\pm$-diagrams directly. Namely, $\mathfrak{S}$ changes
the outer $\pm$-diagram only.
The pair of $\pm$-diagrams corresponding to
$\mathfrak{S}\circ e_1^{\varepsilon_0(P)}\circ\mathfrak{S}(P)$
looks as follows:
\begin{align}\label{pair+-diagram}
\unitlength 10pt
\begin{picture}(36,12.5)
\thicklines
%%%outer
\put(0,0){\line(1,0){36}}
\put(0,0){\line(0,1){11}}
\put(0,11){\line(1,0){8}}
\put(8,11){\line(0,-1){2}}
\put(8,9){\line(1,0){2}}
\multiput(10,9)(0.3,-0.1){10}{\circle*{0.1}}
\put(13,8){\line(1,0){2}}
\put(15,8){\line(0,-1){2}}
\put(15,6){\line(1,0){8}}
\put(23,6){\line(0,-1){2}}
\put(23,4){\line(1,0){2}}
\multiput(25,4)(0.3,-0.1){10}{\circle*{0.1}}
\put(28,3){\line(1,0){2}}
\put(30,3){\line(1,0){2}}
\put(32,3){\line(0,-1){2}}
\put(32,1){\line(1,0){4}}
\put(36,1){\line(0,-1){1}}
%%%inner
\thinlines
\put(0,10){\line(1,0){2}}
\put(2,10){\line(1,0){4}}
%\put(6,10){\line(0,-1){1}}
\put(6,10){\line(1,0){2}}
\put(8,9){\line(0,-1){1}}
\put(8,8){\line(1,0){2}}
%
%\put(13,7){\line(0,-1){1}}
\put(13,7){\line(1,0){2}}
\put(15,6){\line(0,-1){1}}
\put(15,5){\line(1,0){6}}
%\put(21,5){\line(0,-1){1}}
\put(21,5){\line(1,0){2}}
\put(23,4){\line(0,-1){1}}
\put(23,3){\line(1,0){2}}
\put(28,2){\line(1,0){2}}
\put(30,2){\line(0,-1){1}}
\put(30,1){\line(1,0){2}}
\put(32,1){\line(0,-1){1}}
%%%
\put(0.2,10.3){$-$}
\put(1.2,10.3){$-$}
\put(2.2,10.3){$-$}
\put(3.2,10.3){$-$}
\put(4.2,10.3){$-$}
\put(5.2,10.3){$-$}
\put(6.2,10.3){$-$}
\put(7.1,10.3){$-$}
\put(2.2,9.3){$+$}
\put(3.2,9.3){$+$}
\put(4.2,9.3){$-$}
\put(5.2,9.3){$-$}
\put(6.2,9.3){$-$}
\put(7.1,9.3){$-$}
\put(6.2,8.3){$+$}
\put(7.1,8.3){$+$}
\put(8.2,8.3){$-$}
\put(9.2,8.3){$-$}
\put(13.2,7.3){$-$}
\put(14.1,7.3){$-$}
\put(13.2,6.3){$-$}
\put(14.1,6.3){$-$}
\put(13.2,5.3){$+$}
\put(14.1,5.3){$+$}
\put(15.2,5.3){$-$}
\put(16.2,5.3){$-$}
\put(17.2,5.3){$-$}
\put(18.2,5.3){$-$}
\put(19.2,5.3){$-$}
\put(20.2,5.3){$-$}
\put(21.2,5.3){$-$}
\put(22.2,5.3){$-$}
\put(17.2,4.3){$+$}
\put(18.2,4.3){$+$}
\put(19.2,4.3){$-$}
\put(20.2,4.3){$-$}
\put(21.2,4.3){$-$}
\put(22.2,4.3){$-$}
\put(21.2,3.3){$+$}
\put(22.2,3.3){$+$}
\put(23.2,3.3){$-$}
\put(24.2,3.3){$-$}%%%%
\put(28.2,2.3){$-$}
\put(29.2,2.3){$-$}
\put(30.2,2.3){$-$}
\put(31.2,2.3){$-$}
\put(28.2,1.3){$-$}
\put(29.2,1.3){$-$}
\put(30.2,1.3){$+$}
\put(31.2,1.3){$+$}
\put(28.2,0.3){$+$}
\put(29.2,0.3){$+$}
\put(30.2,0.3){$-$}
\put(31.2,0.3){$-$}
\put(32.2,0.3){$-$}
\put(33.2,0.3){$-$}
\put(34.2,0.3){$-$}
\put(35.2,0.3){$-$}
\put(0.5,11.6){$p_r^\cdot$}
\put(2.5,11.6){$p_r^-$}
\put(4.5,11.6){$p_r^+$}
\put(6.0,11.6){$p_{r-2}^\cdot$}
\put(8.5,9.6){$p_{r}^\mp$}
\put(13.5,8.6){$p_i^\cdot$}
\put(15.2,6.6){$p_{i+2}^\mp$}
\put(17.5,6.6){$p_i^-$}
\put(19.5,6.6){$p_i^+$}
\put(21.1,6.6){$p_{i-2}^\cdot$}
\put(23.5,4.6){$p_i^\mp$}
\put(28.6,3.6){$p_1^\cdot$}
\put(30.5,3.6){$p_1^+$}
\put(32.6,1.6){$p_3^\mp$}
\put(34.6,1.6){$p_1^-$}
\end{picture}
\end{align}
Note that the outer shape has also been changed
at $p_1^+$.

\begin{lemma}
The inner $\pm$-diagram of
$e_{\boldsymbol{b}_{2,1}'}\circ
\mathfrak{S}\circ e_1^{\varepsilon_0(P)}\circ\mathfrak{S}(P)$
is of highest type, i.e., each column contains
$+$ as symbol.
\end{lemma}

\begin{proof}
We use Proposition \ref{prop:to highest}.
The quantities $c_i$, $c_i^-$ and $c_i^+$ there
should be used for the corresponding numbers of
the inner $\pm$-diagram of (\ref{pair+-diagram}).
Since we are considering the inner $\pm$-diagram,
we have to understand the word $\boldsymbol{a}$ there
as follows:
\[
\boldsymbol{a}=
2^{a_1}3^{a_2}\cdots (n-1)^{a_{n-2}}n^{a_{n-1}}
(n-1)^{a'_{n-2}}
\cdots 2^{a'_1},
\]
and the formula for $a_i$ and $a_i'$ are the same
in terms of $c_i$, $c_i^-$ and $c_i^+$.
Then,
\begin{align*}
a_1=&\,
\sum_{i=1}^n c_i
+\sum_{i=1}^n c_i^-
-\sum_{i=1}^n c_i^+ -c_1^-\\
=&\,
k+\left\{\sum_{i=3:\mathrm{odd}}^r\left( p_i^++p_{i-2}^\cdot\right)
+p_1^+\right\}
-\left\{\sum_{i=3:\mathrm{odd}}^r\left( p_i^-+p_{i-2}^\cdot\right)
+p_3^\mp +p_1^-\right\}
-p_1^+\\
=&\,
k+\sum_{i=1:\mathrm{odd}}^r
\left( p_i^+-p_i^-\right)-p_3^\mp -p_1^+
\end{align*}
and differences $a_{i+1}-a_i$ are
\begin{align*}
a_{i+1}-a_i=&\,
c_i^+ +c_i^- -c_i -c_{i+1}^-\\
=&\,
\begin{cases}
-p_{i+2}^+ & (\mbox{if }i\mbox{ is odd})\\
-p_{i+3}^\mp & (\mbox{if }i\mbox{ is even and }i\neq r-1)\\
-p_r^\cdot & (i=r-1)
\end{cases}
\end{align*}
and differences $a_i'-a_{i+1}'$ are
\begin{align*}
a_i'-a_{i+1}'=-c_{i+1}^-=
\begin{cases}
-p_{i+2}^+-p_i^\cdot
& (\mbox{if }i\mbox{ is odd})\\
0 & (\mbox{if }i\mbox{ is even}).\\
\end{cases}
\end{align*}
We see that the word $\boldsymbol{a}$
computed here coincides with $\boldsymbol{b}_{2,1}'$.
\end{proof}

\begin{lemma}
$e_1^{p_3^\mp +p_1^-}\circ
e_{\boldsymbol{b}_{2,2}'}\circ
e_{\boldsymbol{b}_{2,1}'}\circ
\mathfrak{S}\circ e_1^{\varepsilon_0(P)}\circ\mathfrak{S}(P)$
is $I_0$-highest.
\end{lemma}

\begin{proof}
Again, we use Proposition \ref{prop:to highest}.
In this case, the quantities $c_i$, $c_i^-$ and $c_i^+$
there mean those for the outer $\pm$-diagram
of (\ref{pair+-diagram}).
Let us compute the word
\[
\boldsymbol{a}=
1^{a_1}2^{a_2}\cdots (n-1)^{a_{n-1}}n^{a_{n}}
(n-1)^{a'_{n-1}}
\cdots 1^{a'_1}
\]
there in the case of our (\ref{pair+-diagram}).
To begin with, $a_1$ is
\begin{align*}
a_1=&\,
\sum_{i=1}^n c_i
+\sum_{i=1}^n c_i^-
-\sum_{i=1}^n c_i^+ -c_1^-
=k+k-p_1^+-\left(p_3^\mp +p_1^-\right)\\
=&\,
2k-p_3^\mp -p_1^- -p_1^+,
\end{align*}
and differences $a_{i+1}-a_i$ are
\begin{align*}
a_{i+1}-a_i=&\,
c_i^+ +c_i^- -c_i -c_{i+1}^-\\
=&\,
\begin{cases}
0 & (\mbox{if }i\mbox{ is odd or }i\geq r)\\
-p_{i+3}^\mp -p_{i+1}^- -p_{i+1}^+ -p_{i-1}^\cdot
& (\mbox{if }i\mbox{ is even and }i\leq r-1)
\end{cases}
\end{align*}
and differences $a_i'-a_{i+1}'$ are
\begin{align*}
a_i'-a_{i+1}'=-c_{i+1}^-=
\begin{cases}
0 & (\mbox{if }i\mbox{ is odd or }i\geq r)\\
-p_{i+3}^\mp -p_{i+1}^- -p_{i+1}^+ -p_{i-1}^\cdot
& (\mbox{if }i\mbox{ is even and }i\neq 2,i\leq r-1)\\
-p_{5}^\mp -p_{3}^- -p_{3}^+ -p_{1}^\cdot -p_1^+
& (\mbox{if }i=2).\\
\end{cases}
\end{align*}
We see that the word $\boldsymbol{a}$
computed here coincides with $\boldsymbol{b}_{2,2}'$
except for $a_1'=p_3^\mp +p_1^-$ which does not appear
in $\boldsymbol{b}_{2,2}'$.
\end{proof}

Let $\tilde{\mu}$ be the $I_0$-highest weight element whose
outer shape coincides with (\ref{pair+-diagram}).
Then the above lemma shows that
$e_{\boldsymbol{b}_2}
e_0^{\varepsilon_0(P)}(P)=f_1^{p_3^\mp +p_1^-}(\tilde{\mu})$.
Since there are exactly $(p_3^\mp +p_1^-)$
columns of height 1 in $\tilde{\mu}$,
we see that the content of columns of height 1 in the tableau
$f_1^{p_3^\mp +p_1^-}(\tilde{\mu})$ are all 2
and that the other columns are the same as $\tilde{\mu}$.
{}From the shape of (\ref{pair+-diagram})
we see that $f_1^{p_3^\mp +p_1^-}(\tilde{\mu})$
coincides with $P'$ given in Proposition \ref{reduction_odd2}.
To summarize, we have
$e_{\boldsymbol{b}_2}
e_0^{\varepsilon_0(P)}(P)=P'$, hence we complete the
proof of Proposition \ref{reduction_odd2}.

We remark that the 1-signature of
$1^l\otimes P'$ is $+^l-^{(p_3^\mp +p_1^-)}+^C$
for some $C$ and from the highest weight condition
of $1^l\otimes P$ we have $l\geq p_3^\mp +p_1^-$.
Thus we cannot apply $e_1$ on $1^l\otimes P'$.

\subsubsection{Reduction to the special case}

\begin{prop} \label{reduction_odd3}
Suppose $\mu,x$ and $P$ are related as in Theorem \ref{th:main}.
Then, with the notions in Propositions \ref{reduction_odd} and \ref{reduction_odd2},
we have 
\begin{itemize}
\item[(i)] $2k-\mu_1+x_1+x_2=\ve_0(P)+l$,
\item[(ii)] $\boldsymbol{b}_1=\boldsymbol{b}_2$, and
\item[(iii)] In view of 
Proposition \ref{reduction_odd},
set $H=H(\mu\ot x)$,
$H'=H(\ol{\mu}\ot\ol{3}^{x_2+\xb_1}1^{l-x_2-\xb_1})$.
Then we have
$H=H'+(x_1+x_2-l)$.
\end{itemize}
\end{prop}

\begin{proof}(i) We have
\begin{align*}
\ve_0(P)=&\,
\sum_{j}(p_j^\cdot +2p_j^++p_j^\mp )-p_1^+
=2k-\sum_{j}(p_j^\cdot +p_j^\mp +2p_j^-)-p_1^+\\
=&\,
2k-
\left(\sum_{j=1:\mathrm{odd}}^{r-2}(\bar{x}_{j+2}-x_{j+2})+x_{r+1}\right)
-\sum_{j=3:\mathrm{odd}}^{r}x_{j-1}
-\left(\sum_{j=3:\mathrm{odd}}^{r}2x_j+2\bar{x}_1\right)\\
&\,
-\left(\mu_1-\bar{x}_1-x_2\right)
=2k-l+x_1-\mu_1+x_2,
\end{align*}
where we have used
$\sum_{j=1}^{r+1}x_j+\sum_{j=1:\mathrm{odd}}^r\bar{x}_j=l$
in the final line.
Thus $\ve_0(P)+l=2k-\mu_1+x_1+x_2$.\\
(ii) To begin with let us show
$\boldsymbol{b}_{1,1}=\boldsymbol{b}_{2,1}$.
We compute
\begin{align*}
t_2=&\,
k+\sum_{j=1:\mathrm{odd}}^r(\mu_j-\bar{x}_j-x_{j+1})
-\left( \bar{x}_1+\sum_{j=3:\mathrm{odd}}^rx_j\right)
-x_2-(\mu_1-\bar{x}_1-x_{2})\\
=&\,
2k-l-\mu_1+x_1,
\end{align*}
thus $t_2+l=s_2$, which shows the coincidence of
the first letters of $\boldsymbol{b}_{1,1}$
and $\boldsymbol{b}_{2,1}$.
As for the other $s_i$ and $t_i$,
note that
\begin{align*}
t_{i+1}-t_i=
\begin{cases}
-p_{i+1}^+=-(\mu_{i+1}-\bar{x}_{i+1}-x_{i+2})
& (\mbox{if } i \mbox{ is even})\\
-p_{i+2}^\mp =-x_{i+1}
& (\mbox{if } i \mbox{ is odd}).
\end{cases}
\end{align*}
When $i$ is odd, we see $t_{i+1}-t_i=s_{i+1}-s_i$.
When $i$ is even, we have
\begin{align*}
s_{i+1}-s_i=&\,
\left(2k-\sum_{j=1:\mathrm{odd}}^{i+1}\mu_j
+x_1+x_{i+2}+\sum_{j=3:\mathrm{odd}}^{i+1}\bar{x}_j\right)\\
&\,
-\left(
2k-\sum_{j=1:\mathrm{odd}}^{i-1}\mu_j
+x_1+\sum_{j=3:\mathrm{odd}}^{i-1}\bar{x}_j
\right)
\end{align*}
and thus we have $t_{i+1}-t_i=s_{i+1}-s_i$,
i.e., $t_i=s_i$ for all $i$.
We have $\beta -t_r=-p_r^\cdot =-x_{r+1}=\alpha -s_r$,
i.e., $\alpha =\beta +l$.
Similarly, we have
$\bar{t}_{r-1}-\beta =
-p_r^+-p_{r-2}^\cdot =
-(\mu_r-\bar{x}_r-x_{r+1})
-(\bar{x}_{r}-x_r)
=-\mu_r+x_r+x_{r+1}
=\bar{s}_{r-1}-\alpha$.
As for other $\bar{t}_i$ and $\bar{s}_i$,
we have
\begin{align*}
\bar{t}_{i+1}-\bar{t}_i=
\begin{cases}
-p_{i+1}^+-p_{i-1}^\cdot
=-(\mu_{i+1}-\bar{x}_{i+1}-x_{i+2})
-(\bar{x}_{i+1}-x_{i+1})
& (\mbox{if } i \mbox{ is even})\\
0
& (\mbox{if } i \mbox{ is odd}).
\end{cases}
\end{align*}
Thus we have $\bar{t}_i+l=\bar{s}_i$ for all $i$
and obtain
$\boldsymbol{b}_{1,1}=\boldsymbol{b}_{2,1}$.

Similarly we can show
$\boldsymbol{b}_{1,2}=\boldsymbol{b}_{2,2}$.
We compute
\begin{align*}
\bar{\bar{t}}_1=
2k-x_2-\bar{x}_1-\left(\mu_1-\bar{x}_1-x_2\right)=
2k-\mu_1,
\end{align*}
thus $\bar{\bar{t}}_1+l=\bar{\bar{s}}_1$,
i.e., the coincidence of the first letters of
$\boldsymbol{b}_{1,1}$ and $\boldsymbol{b}_{2,1}$.
Next, we have $\bar{\bar{t}}_2-(\bar{\bar{t}}_1+l)=-l$.
On the other hand, we have
$\bar{\bar{s}}_1=2k+l-\mu_1$ and
$\bar{\bar{s}}_2=2k-\mu_1$, thus
$\bar{\bar{s}}_2-\bar{\bar{s}}_1=-l$, i.e.,
$\bar{\bar{t}}_2=\bar{\bar{s}}_2$.
Similarly, we can recursively show
$\bar{\bar{t}}_i=\bar{\bar{s}}_i$ for all $i$,
$\beta'=k$,
$\bar{\bar{\bar{t}}}_i=\bar{\bar{\bar{s}}}_i$
for all $i$.
So we have
$\boldsymbol{b}_{1,2}=\boldsymbol{b}_{2,2}$,
and therefore we get the final result
$\boldsymbol{b}_{1}=\boldsymbol{b}_{2}$.\\
(iii) The 0-signature of $\mu\otimes x$ is
$-^{2k-\mu_1}\cdot -^{x_1+x_2}+^{C_1}$ for some $C_1$
and that of $1^l\otimes P$ is
$-^l\cdot -^{2k-l-\mu_1+x_1+x_2}+^{C_2}$ for some $C_2$.
Here we divide into two cases.
Let us first assume $2k-\mu_1\geq l$.
Then the actions of $e_0$ 
on two tensor products look as follows
(proceed from left to right):
\begin{align}
\mu\otimes x:\,&
\underbrace{R\cdots\cdots R}_{x_1+x_2}
\underbrace{L\cdots\cdots\cdots\cdots\cdots L}_{2k-\mu_1}
\nonumber\\
1^l\otimes P :\,&
\underbrace{R\cdots\cdots\cdots\cdots\cdots R}_{2k-l-\mu_1+x_1+x_2}
\underbrace{L\cdots\cdots L}_{l}
\nonumber
\end{align}
Thus we have $(x_1+x_2)$ (RR) pairs and
$l$ (LL) pairs.
Therefore we have
$H'=H-(x_1+x_2)+l$ which gives the desired relation.
Next assume $2k-\mu_1\leq l$.
Then we have $(2k-l-\mu_1+x_1+x_2)$ (RR) pairs
and $(2k-\mu_1)$ (LL) pairs and again we obtain
$H'=H-(x_1+x_2)+l$.
\end{proof}

\subsection{Even $r$ case}
Since the proofs are similar to those
for the odd $r$ case, we only describe the results.

\subsubsection{Calculation in $B^{r,k}\otimes B^{1,l}$}
Let $\mu\ot x\in B^{r,k}\ot B^{1,l}$ be $I_0$-highest and $\mu=\sum_i\mu_i\La_i$.
$\mu_i=0$ unless $1\le i\le r$ and $i$ is even. We also know that the coordinates
other than $x_1,x_2,\dots,x_{r+1},\bar{x}_r,\ldots,\bar{x}_4,\bar{x}_2$ are $0$ by
Proposition \ref{prop:ht cond 1}.
Let us set $c=(x_1-\mu_0)_+$ throughout this subsection.

Let us define a word
$\boldsymbol{b}_3=\boldsymbol{b}_{3,1}\boldsymbol{b}_{3,2}$ by
\begin{align}
\boldsymbol{b}_{3,1}=
&\,
2^{s_2}3^{s_3}\cdots r^{s_r}
(r+1)^{\alpha}(r+2)^{\alpha}\cdots n^{\alpha}(n-1)^{\alpha}
\cdots (r+1)^{\alpha}r^{\alpha}
\nonumber\\
&
(r-1)^{\bar{s}_{r-1}}(r-2)^{\bar{s}_{r-2}}\cdots
2^{\bar{s}_{2}},
\nonumber\\
\boldsymbol{b}_{3,2}=
&\,
1^{\bar{\bar{s}}_1}
2^{\bar{\bar{s}}_2}3^{\bar{\bar{s}}_3}\cdots
(r-1)^{\bar{\bar{s}}_{r-1}}
r^k(r+1)^k\cdots n^k(n-1)^k\cdots (r+1)^kr^k
\nonumber\\
&
(r-1)^{\bar{\bar{\bar{s}}}_{r-1}}
(r-2)^{\bar{\bar{\bar{s}}}_{r-2}}\cdots
2^{\bar{\bar{\bar{s}}}_{2}}
1^{\bar{\bar{\bar{s}}}_{1}}
\nonumber
\end{align}
where the exponents are defined as follows.
For $i=1,2,\cdots,r/2$,
\begin{align}
s_{2i}=&\,
2k-\sum_{j=0:\mathrm{even}}^{2i}\mu_j
+c+x_{2i+1}+\sum_{j=2:\mathrm{even}}^{2i}\bar{x}_{j},
\qquad
s_{2i+1}=s_{2i}-x_{2i+1}.
\nonumber
\end{align}
Define $\alpha=k+c+\sum_{j=2:\mathrm{even}}^{r}\bar{x}_{j}$.
Then $s_r=\alpha+x_{r+1}$.
For $i=r/2-1,\cdots,2,1$,
\begin{align}
\bar{s}_{2i+1}=\bar{s}_{2i}=
\alpha-\sum_{j=2i+2:\mathrm{even}}^r\mu_j
+\sum_{j=2i+2}^{r+1}x_j.
\nonumber
\end{align}
Set $\bar{\bar{s}}_1=2k+l-x_1+c$
and define other $\bar{\bar{s}}_i$ by
$\bar{\bar{s}}_{2i}=\bar{\bar{s}}_{2i+1}=
2k-\sum_{j=0:\mathrm{even}}^{2i}\mu_j$
for $i=1,2,\cdots,r/2-1$.
Note that $\bar{\bar{s}}_{r-1}=k+\mu_r$.
Set ${\bar{\bar{\bar{s}}}_{1}}=0$
and define other
${\bar{\bar{\bar{s}}}_{i}}$ by
${\bar{\bar{\bar{s}}}_{2i-1}}=
{\bar{\bar{\bar{s}}}_{2i-2}}=
k-\sum_{j=2i:\mathrm{even}}^r\mu_j$
for $i=r/2-1,\cdots,3,2$.

Then we have:
\begin{prop}\label{reduction_even}
We have
$$e_{\boldsymbol{b}_3}e_0^{2k-\mu_0+c}(\mu\otimes x)
=\bar{\mu}\otimes 1^{l}$$
where $\bar{\mu}=\sum_{i=2:\mathrm{even}}^r\mu_i\La_i+\mu_0\La_2$.
\end{prop}

We remark that $\varepsilon_0(\mu)=2k-\mu_0$ and
$\varphi_0(\mu)=\mu_0$.
This nonzero $\varphi_0(\mu)$ is the origin of $c$
in the above formula.

\subsubsection{Calculation in $B^{1,l}\otimes B^{r,k}$}
In this subsection, let $P$ and $P'$ be the $\pm$-diagrams.
As before, corresponding to $P$ and $P'$,
we use the parametrization $p_i^\ast$ and
$p_i'{}^\ast$ ($\ast =\cdot,+,-,\mp$) respectively.
Define a word $\boldsymbol{b}_4=\boldsymbol{b}_{4,1}
\boldsymbol{b}_{4,2}$ by
\begin{align}
\boldsymbol{b}_{4,1}=
&\,
2^{t_2+l}3^{t_3+l}\cdots r^{t_r+l}
(r+1)^{\beta +l}(r+2)^{\beta +l}\cdots n^{\beta +l}(n-1)^{\beta +l}
\cdots (r+1)^{\beta +l}r^{\beta +l}
\nonumber\\
&
(r-1)^{\bar{t}_{r-1}+l}(r-2)^{\bar{t}_{r-2}+l}\cdots
2^{\bar{t}_{2}+l},
\nonumber\\
\boldsymbol{b}_{4,2}=
&\,
1^{\bar{\bar{t}}_1+l}
2^{\bar{\bar{t}}_2}3^{\bar{\bar{t}}_3}\cdots
(r-1)^{\bar{\bar{t}}_{r-1}}
r^{\beta'}(r+1)^{\beta'}\cdots
n^{\beta'}(n-1)^{\beta'}\cdots (r+1)^{\beta'}r^{\beta'}
\nonumber\\
&
(r-1)^{\bar{\bar{\bar{t}}}_{r-1}}
(r-2)^{\bar{\bar{\bar{t}}}_{r-2}}\cdots
2^{\bar{\bar{\bar{t}}}_{2}}.
%1^{\bar{\bar{\bar{t}}}_{1}}.
\nonumber
\end{align}
Here the exponents for $\boldsymbol{b}_{4,1}$
are
\begin{align*}
&t_2=k+\sum_{i=2:\mathrm{even}}^r
(p_i^+-p_i^-)-p_2^+-p_2^\mp ,\\
&t_{i+1}-t_i=
\begin{cases}
-p_{i+1}^+ & (\mbox{if } i \mbox{ is odd})\\
-p_{i+2}^\mp & (\mbox{if } i \mbox{ is even}),
\end{cases}\\
&\beta -t_r=-p_r^\cdot ,\qquad
\bar{t}_{r-1}-\beta =-p_r^+-p_{r-2}^\cdot ,\\
&\bar{t}_{i}-\bar{t}_{i+1}=
\begin{cases}
-p_{i-1}^+-p_{i+1}^\cdot & (\mbox{if } i \mbox{ is odd})\\
0 & (\mbox{if } i \mbox{ is even}).
\end{cases}
\end{align*}
The exponents for $\boldsymbol{b}_{4,2}$
are
\begin{align*}
&\bar{\bar{t}}_1=2k-p^\mp_2\\
&\bar{\bar{t}}_{i+1}-\bar{\bar{t}}_i=
\begin{cases}
-p^\mp_{i+3}-p^-_{i+1}-p^+_{i+1}-p^\cdot_{i-1}
& (\mbox{if } i \mbox{ is odd})\\
0 & (\mbox{if } i \mbox{ is even}),
\end{cases}\\
& \beta'-\bar{\bar{t}}_{r-1}=
\bar{\bar{\bar{t}}}_{r-1}-\beta'=
-p^\cdot_{r}-p^-_{r}-p^+_{r}-p^\cdot_{r-2},\\
& \bar{\bar{\bar{t}}}_{i}-\bar{\bar{\bar{t}}}_{i+1}=
\begin{cases}
-p^\mp_{i+3}-p^-_{i+1}-p^+_{i+1}-p^\cdot_{i-1}
& (\mbox{if } i \mbox{ is odd and }i\neq 1)\\
-p^\mp_{4}-p^-_{2}-p^+_{2}-p_2^\mp -p^\cdot_{0}
& (\mbox{if } i=1)\\
0 & (\mbox{if } i \mbox{ is even}).
\end{cases}
\end{align*}

Then the result is:
\begin{prop}\label{reduction_even2}
We have 
\begin{align}
e_{\boldsymbol{b}_4}e_0^{\varepsilon_0(P)+l}(1^l\otimes P)
=1^l\otimes P'
\nonumber
\end{align}
where
\begin{align*}
{\varepsilon_0(P)}=&\,
p_r^\cdot +2p_r^+ +p_{r-2}^\cdot+
\sum_{j=2:\mathrm{even}}^{r-2}
(p^\mp_{j+2} +2p^+_j+p^\cdot_{j-2}),
\end{align*}
and $P'$ is related with $P$ as
\begin{align*}
&p_2'{}^+=p_4^\mp +p_2^- +p_2^+ +p_2^\mp +p_0^\cdot,&
&p_r'{}^+=p_r^\cdot +p_r^-+p_r^+ +p_{r-2}^\cdot ,\\
&p_i'{}^+=p_{i+2}^\mp +p_i^-+p_i^++p_{i-2}^{\cdot},
\end{align*}
where $i$ is an even integer such that $2<i<r$
and all other $p_i'{}^\ast =0$.
\end{prop}

Obviously, $P'$ above is an $I_0$-highest weight element.

\subsubsection{Reduction to the special case}

\begin{prop} \label{reduction_even3}
Suppose $\mu,x$ and $P$ are related as in Theorem \ref{th:main}.
Then, with the notions in Propositions \ref{reduction_even} and \ref{reduction_even2},
we have 
\begin{itemize}
\item[(i)] $2k-\mu_0+c=\ve_0(P)+l$,
\item[(ii)] $\boldsymbol{b}_3=\boldsymbol{b}_4$, and
\item[(iii)] $H(\mu\ot x)=(x_1-\mu_0)_+-l$.
\end{itemize}
\end{prop}


\begin{thebibliography}{99}

\bibitem{FOS}
G.~Fourier, M.~Okado, A.~Schilling,
\textit{Kirillov-Reshetikhin crystals for nonexceptional types}, 
preprint arXiv:0810.5067. 

\bibitem{FOS2}
G.~Fourier, M.~Okado, A.~Schilling,
\textit{Perfectness of Kirillov-Reshetikhin crystals for nonexceptional types}, 
preprint arXiv:0811.1604. 

%\bibitem{FSS:2007}
%G.~Fourier, A.~Schilling, M.~Shimozono,
%\textit{Demazure structure inside Kirillov-Reshetikhin crystals},
%J. Algebra \textbf{309} (2007) 386--404.

\bibitem{FOY}
K.~Fukuda, M.~Okado, Y.~Yamada,
\textit{Energy functions in box ball systems}, 
Int. J. Mod. Phys. A {\bf 15} (2000) 1379--1392. 

\bibitem{HHIKTT}
G.~Hatayama, K.~Hikami, R.~Inoue, A.~Kuniba, T.~Takagi, T.~Tokihiro,
\textit{The $A_M^{(1)}$ automata related to crystals of symmetric tensors},
J. Math. Phys. {\bf 42} (2001) 274--308. 

%\bibitem{HKOTT:2002}
%G.~Hatayama, A.~Kuniba, M.~Okado, T.~Takagi, Z.~Tsuboi,
%\textit{Paths, crystals and fermionic formulae}, MathPhys Odyssey 2001, 205--272, 
%Prog.\ Math.\ Phys.\ {\bf 23}, Birkh\"auser Boston, Boston, MA, 2002.

%\bibitem{HKOTY:1999} 
%G.~Hatayama,  A.~Kuniba, M.~Okado, T.~Takagi, Y.~Yamada, 
%\textit{Remarks on fermionic formula},
%Contemporary Math.\ {\bf 248} (1999) 243--291.

\bibitem{HKOTY2} 
G.~Hatayama,  A.~Kuniba, M.~Okado, T.~Takagi, Y.~Yamada, 
\textit{Scattering rules in soliton cellular automata associated with crystal bases},
Contemporary Math.\ {\bf 297} (2002) 151--182.

\bibitem{HK:2002}
J.~Hong, S.-J.~Kang,
Introduction to quantum groups and crystal bases,
Graduate Studies in Mathematics, \textbf{42}, American Mathematical Society, Providence, RI, 2002. xviii+307 pp. 

\bibitem{Kac}
V. G. Kac, 
\textit{``Infinite Dimensional Lie Algebras,"}
3rd ed., Cambridge Univ. Press, Cambridge, UK, 1990.

\bibitem{KKM}
S.-J.~Kang, M.~Kashiwara, K.~C.~Misra,
\textit{Crystal bases of Verma modules for quantum affine Lie algebras},
Compositio Math. {\bf 92} (1994) 299--325.

\bibitem{KMN1:1992}
S-J.~Kang, M.~Kashiwara, K.~C.~Misra, T.~Miwa, T.~Nakashima, A.~Nakayashiki,
\textit{Affine crystals and vertex models},
Int.\ J.\ Mod.\ Phys.\ A {\bf 7} (suppl. 1A) (1992), 449--484.

%\bibitem{KMN2:1992}
%S.-J.~Kang, M.~Kashiwara, K.~C.~Misra, T.~Miwa, T.~Nakashima, A.~Nakayashiki,
%\textit{Perfect crystals of quantum affine Lie algebras},
%Duke Math.\ J.\ \textbf{68} (1992) 499--607.

\bibitem{Ka:1991}
M.~Kashiwara,
\textit{On crystal bases of the $q$-analogue of universal enveloping algebras},
Duke Math.\ J.\ {\bf 63} (1991), 465--516.

\bibitem{Ka3}
M.~Kashiwara, 
\textit{On level zero representations of quantized affine algebras},
Duke Math.\ J.\ {\bf 112} (2002) 117--175.

\bibitem{KN:1994}
M.~Kashiwara, T.~Nakashima,
\textit{Crystal graphs for representations of the $q$-analogue
of classical Lie algebras},
J. Alg. \textbf{165} (1994) 295--345.

\bibitem{KKR}
S.~V.~Kerov, A.~N.~Kirillov, N.~Yu.~Reshetikhin,
\textit{Combinatorics, the Bethe ansatz and representations of the symmetric group},
Zap.Nauchn. Sem. (LOMI) {\bf 155} (1986) 50--64
(English translation: J. Sov. Math. {\bf 41} (1988) 916--924).

\bibitem{KR1}
A.~N.~Kirillov, N.~Yu.~Reshetikhin,
\textit{The Bethe ansatz and the combinatorics of Young tableaux},
J. Sov. Math. {\bf 41} (1988) 925--955.

\bibitem{KSS}
A.~N. Kirillov, A.~Schilling, M.~Shimozono,
\textit{A bijection between Littlewood-Richardson tableaux and rigged configurations}, 
Selecta Math. (N.S.) {\bf 8} (2002) 67--135.

\bibitem{O}
M.~Okado,
\textit{$X=M$ conjecture}, MSJ Memoirs \textbf{17} (2007) 43--73.

\bibitem{O:2007}
M.~Okado,
\textit{Existence of crystal bases for Kirillov-Reshetikhin modules of type $D$},
Publ. RIMS \textbf{43} (2007) 977-1004.

\bibitem{OS:2008}
M.~Okado, A.~Schilling,
\textit{Existence of Kirillov-Reshetikhin crystals for nonexceptional types},
Representation Theory \textbf{12} (2008) 186--207.

\bibitem{Sa}
R.~Sakamoto,
\textit{Kirillov--Schilling--Shimozono bijection as energy functions of crystals},
Internat. Math. Res. Notices (2009) 579--614.

\bibitem{S:2008}
A.~Schilling,
\textit{Combinatorial structure of Kirillov-Reshetikhin crystals of type $D_n^{(1)}$, $B_n^{(1)}$, 
$A_{2n-1}^{(2)}$},
J. Algebra \textbf{319} (2008) 2938--2962.

\end{thebibliography}
\end{document}